%% file: main.tex
\documentclass[oldfontcommands, 11pt, reqno]{amsart}  

\input{macros}
\usepackage{verbatim}  
\usepackage{vector}  
\usepackage[top = 1in, bottom = 1in, left = 1in, right = 1in]{geometry}
\usepackage[backend=biber,
            isbn=false,
            doi=false,
            url=false,
            maxbibnames=5,
            style=alphabetic,
            citestyle=alphabetic]{biblatex}
\renewbibmacro{in:}{}
\usepackage{hyperref}
\hypersetup{linkcolor=blue, urlcolor=blue, colorlinks=true, linktocpage=true, citebordercolor = 0 1 0}  
\usepackage{cleveref}
\begin{document}


\title  {A Classical Bulk-Boundary Correspondence}
\author {Eugene Rabinovich}
\address{Department of Mathematics\\
University of Notre Dame\\
255 Hurley\\
Notre Dame, IN 46556}
\email{erabinov@nd.edu}
\maketitle
\epigraph{That which is below is from that which is above.}{\emph{Emerald Table,} Holmyard translation}

\input{abstract}







\pagestyle{headings}  
\tableofcontents
\section{Introduction}
In this paper, we articulate and prove---in the language of $\PO$-factorization algebras---a correspondence between ``degenerate'' (in the sense that Poisson structures are ``degenerate'' symplectic structures) classical field theories on a manifold $N$ and their universal bulk-boundary systems on $N\times \RRge$. Our aim is to provide a justification and generalization of the basic insight that led Kontsevich to his deformation quantization of Poisson manifolds \autocite{KontPSM}.

\subsection{Background}
Cattaneo and Felder \autocite{CFPSM} have given an interpretation of Kontsevich's deformation quantization procedure in terms of correlation functions in a two-dimensional field theory defined on the unit disk.
The field theory is known as the \emph{Poisson sigma model} \autocite{schallerstrobl, ikeda} and the observables whose correlation functions are computed ``have support on'' the boundary circle.
One of the most surprising things about this fact is the need to pass to a two-dimensional theory to obtain these correlation functions, even if the correlation functions themselves ``live on'' a one-dimensional space.
By contrast, for symplectic manifolds, it is known \autocite{gradylili} that Fedosov's analogous procedure \autocite{fedosov} for deformation quantization corresponds to the quantization of a one-dimensional Batalin-Vilkovisky (BV) \autocite{originalBV} theory.
One wonders, therefore, whether there is a one-dimensional theory one may associate to a Poisson manifold.
The answer is yes, but it is a Poisson BV theory, as opposed to an ``ordinary'' symplectic BV theory.
(Perturbative) Poisson BV theories were defined by Butson and Yoo \autocite{butsonyoo}, who called them ``degenerate field theories.''
``Ordinary'' BV theories are described in the language of $(-1)$-shifted symplectic geometry; Poisson BV theories are described, as the name suggests, in the language of the analogous \emph{Poisson} geometry.
Henceforth, we will refer to ``ordinary'' BV theories as ``symplectic BV theories'' or simply ``BV theories'' when the risk of confusion is not high; we will always refer to Poisson BV theories by the full term.
Furthermore, whenever the term ``theory'' appears in this paper, it is to be understood to mean ``perturbative theory''.

For any PBV theory $\cT$ on a manifold $N$ (without boundary), Butson and Yoo describe a BV theory $\cZ(\cT)$ on $N\times \RR_{\geq 0}$.
This BV theory is called the ``universal bulk theory'' for $\cT$.
Butson and Yoo call the one-dimensional PBV theory implicit in Kontsevich's work ``topological Poisson mechanics.''
Its universal bulk theory is precisely the Poisson sigma model.
The universal bulk theory comes with a canonical boundary condition which corresponds to $\cT$.
Together, $\cZ(\cT)$ and $\cT$ determine what one may call a bulk-boundary system \autocite{classicalarxiv}.
In his PhD thesis \autocite{ERthesis}, the author has developed some techniques for the study of classical and quantum aspects of bulk-boundary systems.
The main theorem of this paper relates the Poisson BV theory theory $\cT$ on $N$ (as described by Butson and Yoo) to the bulk-boundary system $(\cZ(\cT),\cT)$ on $N\times \RRge$ (as described by the author).
Our work builds most directly on that of Butson and Yoo, although we are aware of similar constructions in \autocite{TJF1} and \autocite{TJF2}.

To articulate this main theorem more concretely, we use the language of (pre)factorization algebras, as in the work of Costello and Gwilliam \autocite{CG1}. 
A prefactorization algebra on a space $M$ is in particular a precosheaf on $M$; however, it possesses further structure maps, namely maps which allow one to ``multiply'' elements assigned to disjoint open subsets of $M$.
One may also consider a descent condition for prefactorization algebras, analogous to codescent for cosheaves; objects satisfying this additional condition are called factorization algebras.
Factorization algebras model the structure present in the classical and quantum observables of BV theories.
Indeed, Costello and Gwilliam \autocite{cost, CG1, CG2} have shown how to obtain a factorization algebra of observables from a classical or quantum BV theory on a manifold $M$ without boundary.
Moreover, in the classical case, the relevant factorization algebra possesses a cohomological degree $+1$ Poisson bracket, so the factorization algebra of classical observables is, in fact, a $\PO$-factorization algebra.
(There are various ways to say what this means; in brief, such a factorization algebra $\cF$ assigns to each open set a $\PO$-algebra, and the structure maps are compatible with this structure.)
This classical case has been extended in two different directions. First, given a Poisson BV theory $\cT$ on a manifold $N$, there is a $\PO$-factorization algebra $\Obcl_{\cT}$ on $N$ of classical observables for the Poisson BV theory \autocite{butsonyoo}.
Second, given a bulk-boundary system $(\sE,\sL)$ (here $\sE$ denotes the ``bulk theory'' and $\sL$ the ``boundary condition'') on a manifold $M$ with boundary, the author has shown \autocite{classicalarxiv} that there is a $\PO$-factorization algebra $\Obcl_{\sE,\sL}$ on $M$ of classical observables of the bulk-boundary system.
For a Poisson BV theory on a manifold $N$, we therefore obtain two $\PO$-factorization algebras: $\Obcl_{\cT}$, which lives on $N$, and $\Obcl_{\cZ(\cT),\cT}$, which lives on $N\times \RR_{\geq 0}$.
In this paper, we explore the relationship between these two factorization algebras.

\subsection{Statement and Interpretation of Results}

To state the main theorem of this paper, we need to introduce a bit of notation.
$\PO$-prefactorization algebras are algebras over a particular colored operad, which we denote $\PO\text{-}\Disj_N$.
We let $\hoPODisj$ denote the bar-cobar resolution of this colored operad.
(The subscript $_\infty$ is meant to indicate that $\hoPODisj$ is a cofibrant resolution of $\PO\text{-}\Disj_N$; the hat indicates that the resolution is not obtained using Koszul duality theory.)
We describe the operads $\PO\text{-}\Disj$ and $\hoPODisj$ in more detail in Section \ref{subsec: operadconventions}; here, we note only that there is a quasi-isomorphism of colored operads $\hoPODisj\to \PO\text{-}\Disj_N$, so that any $\PO$-prefactorization algebra is in particular an algebra over $\hoPODisj$ and the categories of algebras over the two operads are ``the same'' in a suitable sense.

The main theorem of this paper is:

\begin{theorem}
\label{thm: main}
Let $\cT$ be a Poisson BV theory on $N$, and $\cZ(\cT)$ its universal bulk theory. Let $\rho:N\times \RR_{\geq 0} \to N$ denote the projection onto the first component. Then, there is an $\infty$-quasi-isomorphism of $\hoPODisj$-algebras
\[
\Obcl_{\cT} \overset{\sim}{\to} \rho_* \left(\Obcl_{\cZ(\cT),\cT}\right)
\]
on~$N$,
where $\rho_*$ denotes the pushforward of factorization algebras from $N\times \RRge$ to $N$. 
\end{theorem}

Theorem \ref{thm: main} is the bulk-boundary correspondence appearing in the title of this article.
Namely, the correspondence relates a $\PO$-factorization algebra on $N\times \RRge$ to one on the boundary $N$.
This theorem is a generalization of the main theorem of \autocite{GRW}.
In that paper, the authors consider only Poisson BV theories which induce \emph{free} bulk-boundary systems.
In that case, one obtains a \emph{strict} quasi-isomorphism of $\PO$-factorization algebras, so the algebraic machinery of operads is necessary for neither the statement nor the proof of the main theorem therein.

We note that Theorem \ref{thm: main} applies for any Poisson BV theory $\cT$ on $N$. 
In particular, $\cT$ may depend on arbitrarily complicated geometry on $N$; there is no need for the theory to be topological in nature.
The universal bulk-boundary system for $\cT$ is, nevertheless, topological along $\RRge$, even if it may depend on complicated geometry on $N$.
We explore a few of these non-topological possibilities in Section \ref{sec: PBV}.

\begin{figure}[h]
    \centering
    \begin{tikzpicture}
            \node at (0,0) (a){};
             \draw[rounded corners] (0,0) arc (270:400: 1 cm and 1.75 cm).. controls +(-.15,.5)..++(0,1) arc (-40:90: 1cm and 1.75cm) node(b){};
             \draw[rounded corners] (b) arc (90: 220: 1cm and 1.75 cm)..controls +(.15,-.5)..++(0,-1) arc (140: 270: 1cm and 1.75 cm);
             \path (a.center) -- +(0,.75) node (c){};
             \draw (c.center) .. controls +(.25,.875).. ++(0,1.75) node[pos = .1](d){} node[pos = .9](e){};
             \draw (d.center) to[bend left] (e.center);
             \path (b.center) -- +(0,-.75) node (f){};
             \draw (f.center) ..controls +(.25,-.875).. ++(0,-1.75) node[pos =.1](g){} node[pos = .9](h){};
             \draw (h.center) to[bend left] (g.center);
            \path (a.center) --(b.center) node[pos = 0.5](mid){};
             
            \node at (05,0) (a1){};
             \draw[dashed,rounded corners] (a1) arc (270:400: 1 cm and 1.75 cm).. controls +(-.15,.5)..++(0,1) arc (-40:90: 1cm and 1.75cm) node(b1){};
             \draw[dashed, rounded corners] (b1) arc (90: 220: 1cm and 1.75 cm)..controls +(.15,-.5)..++(0,-1) arc (140: 270: 1cm and 1.75 cm);
             \path (a1.center) -- +(0,.75) node (c1){};
             \draw[dashed] (c1.center) .. controls +(.25,.875).. ++(0,1.75) node[pos = .1](d1){} node[pos = .9](e1){};
             \draw[dashed] (d1.center) to[bend left] (e1.center);
             \path (b1.center) -- +(0,-.75) node (f1){};
             \draw[dashed] (f1.center) ..controls +(.25,-.875).. ++(0,-1.75) node[pos =.1](g1){} node[pos = .9](h1){};
             \draw[dashed] (h1.center) to[bend left] (g1.center);
             \draw[->] (a.center)--++(7,0);
             \draw[->] (b.center)--++(7,0);
             
            \node at (-3,0) (a2){};
             \draw[rounded corners] (a2) arc (270:400: 1 cm and 1.75 cm).. controls +(-.15,.5)..++(0,1) arc (-40:90: 1cm and 1.75cm) node(b2){};
             \draw[rounded corners] (b2) arc (90: 220: 1cm and 1.75 cm)..controls +(.15,-.5)..++(0,-1) arc (140: 270: 1cm and 1.75 cm);
             \path (a2.center) -- +(0,.75) node (c2){};
             \draw (c2.center) .. controls +(.25,.875).. ++(0,1.75) node[pos = .1](d2){} node[pos = .9](e2){};
             \draw (d2.center) to[bend left] (e2.center);
             \path (b2.center) -- +(0,-.75) node (f2){};
             \draw (f2.center) ..controls +(.25,-.875).. ++(0,-1.75) node[pos =.1](g2){} node[pos = .9](h2){};
             \draw (h2.center) to[bend left] (g2.center);
             \path (a2.center) -- (b2.center) node[pos = 0.5](mid2){};
             \path (mid.center) -- (mid2.center) node[pos = 0.4](arrowstart){} node[pos=0.6](arrowend){};
             \draw[->] (arrowstart.center)--(arrowend.center) node[pos=.5, above]{$\rho$};

             \path (a2.center) -- ++(-1,-.3) node (leftbrace1){};
             \draw[snake=brace, mirror snake] (leftbrace1.center)--++(2,0) node[pos=.1, text width = 0.25cm, align=center, below]{
             \begin{varwidth}{1cm}
             \begin{align*}
                \mathrm{O}&\mathrm{bs}^{cl}_\cT\\
                &\text{\rotatebox[origin =c]{90}{$\simeq$}}\\
                \rho_*\mathrm{O}&\mathrm{bs}^{cl}_{\cZ(\cT),\cT}
             \end{align*}
             \end{varwidth}
             };

             \path (a.center) -- ++(-1,-.3) node (leftbrace2){};
             \draw[snake=brace, mirror snake] (leftbrace2.center)--++(8.25,0) node[pos=.4, text width = 0.25cm, align=center, below]{
             \begin{varwidth}{1cm}
             \begin{align*}
                &\Obcl_{\cZ(\cT),\cT}
            \end{align*}
            \end{varwidth}
             };
    \end{tikzpicture}
    \caption{A schematic of Theorem \ref{thm: main}. The theorem relates two factorization algebras on the boundary, one of which originates on the half-cylinder on the right.}
    \label{fig: schema}
\end{figure}
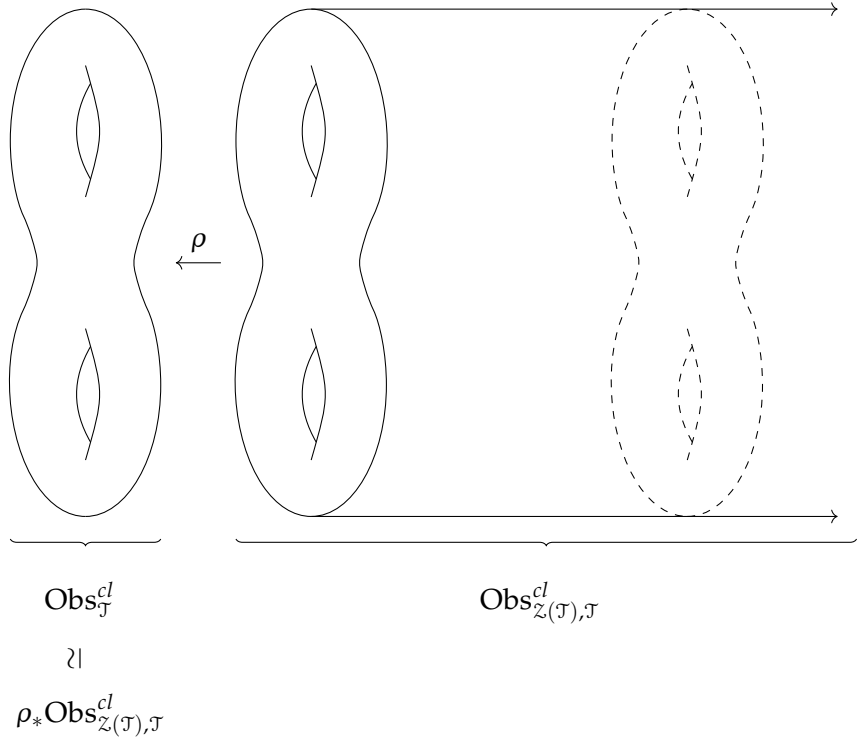

Figure \ref{fig: schema} gives a schematic of Theorem \ref{thm: main}. 
Let us give some words of interpretation for the theorem and its accompanying picture.
One should think of $\rho_*\left(\Obcl_{\cZ(\cT),\cT}\right)$ as the ``boundary'' part of the observables for the bulk-boundary system.
Theorem \ref{thm: main} therefore describes the way in which the $\PO$-factorization algebra $\Obcl_{\cZ(\cT),\cT}$ knows everything about the $\PO$-prefactorization algebra $\Obcl_\cT$.
Note that by considering the behavior of the factorization algebra $\Obcl_{\cZ(\cT),\cT}$ on the interior $N\times \RR_{>0}$, we presumably obtain more information about $\Obcl_{\cZ(\cT),\cT}$, namely ``pure bulk'' information.
The full factorization algebra $\Obcl_{\cZ(\cT),\cT}$ also encodes what may be understood as an ``action'' of the bulk observables on the boundary observables.
We will conjecture more about this ``pure bulk'' and ``fully bulk-boundary'' information below.
Namely, we will claim that the ``fully bulk-boundary'' information is determined from $\Obcl_\cT$ by a universal property.

For the reader already familiar with factorization algebras, it may seem somewhat surprising that the pushforward $\rho_*(\Obcl_{\cZ(\cT),\cT})$ is described as the boundary part of the observables, since this pushforward assigns to an open subset $U\subseteq N$ the space $\Obcl_{\cZ(\cT),\cT}(U\times \RRge)$ of observables on the whole half-cylindrical set $U\times \RRge$.
To address this potential issue, we note that $\Obcl_{\cZ(\cT),\cT}$ is stratified locally constant along $\RRge$; in other words, for any open $U\subseteq N$ and any inclusion of open intervals $I_1\subseteq I_2\subset \RRge$, the structure map
\[
\Obcl_{\cZ(\cT),\cT}(U\times I_1)\to \Obcl_{\cZ(\cT),\cT}(U\times I_1)
\]
is a quasi-isomorphism.
The same statement is true if $I_1$ and $I_2$ are both intervals of the form $[0,a)$.
Hence, we need not make much distinction between $\Obcl_{\cZ(\cT),\cT}(U\times [0,\epsilon))$ and $\Obcl_{\cZ(\cT),\cT}(U\times \RRge)$ for any $\epsilon>0$.
In fact, as $U$ varies over the open subsets of $N$, the two resulting factorization algebras on $N$ are naturally equivalent; the factorization algebra resulting from the latter object is precisely $\rho_*\left(\Obcl_{\cZ(\cT),\cT}\right)$, and the former object--in the limit $\epsilon\to 0$--can be properly called the ``boundary observables.''

Theorem \ref{thm: main} should not be at all surprising from a ``zoomed-out,'' schematic standpoint, as we now explain.
We emphasize that the following discussion is schematic because there are many subtleties in defining non-degeneracy conditions for shifted symplectic and coisotropic structures on infinite-dimensional spaces.
Nevertheless, it is illustrative to ignore these subtleties for the moment.
The universal bulk-boundary system of $\cT$ comes with a sheaf $\sE$ of bulk fields on $N\times\RRge$, a 0-shifted symplectic sheaf of boundary fields $\sEb$ on $N$, and a sheaf $\sL$ on $N$ corresponding to the boundary condition.
There are coisotropic maps of sheaves
\[
\xymatrix{
& \sE\ar[d]\\
\mathrm{inc}_*\sL\ar[r] & \mathrm{inc}_*\sEb,
}
\]
where $\mathrm{inc}$ is the inclusion $N\times\{0\}\to N\times \RRge$.
Open-by-open on $N\times \RRge$, the factorization algebra $\Obcl_{\cZ(\cT),\cT}$ is essentially the space of functions on the homotopy pullback $\sE\times^h_{\mathrm{inc}_*\sEb}\mathrm{inc}_*\sL$, and the $\PO$-algebra structure it acquires as a result is from the shifted Poisson structure on this homotopy pullback \autocite{safronovpoissreduction}.
In a general bulk-boundary system, there is a Poisson BV structure on the boundary condition $\sL$ induced from taking the coisotropic intersection of $\sL$ with $\sEb$ inside $\sEb$; Butson and Yoo spell this out in detail.
Moreover, the universal bulk-boundary system of a Poisson BV theory is constructed in such a way that the induced Poisson BV structure on $\sL$ it acquires from being a boundary condition for this theory coincides with the original Poisson BV structure.
Moreover, on open subsets of the form $U\times \RRge$, where $U\subseteq N$, the map $\sE\to \mathrm{inc}_*\sEb$ is an equivalence.
Hence 
\[
\sE(U\times \RRge)\times^h_{\sEb(U)}\mathrm\sL(U) \simeq \sL(U)\simeq \sEb(U)\times^h_{\sEb(U)}\sL(U);
\]
on an open subset $U\subseteq N$, $\rho_*\Obcl_{\cZ(\cT),\cT}(U)$ is the space of functions on the first homotopy pullback, while $\Obcl_{\cT}(U)$ is the space of functions on this second homotopy pullback, and so the immediately preceding equivalences of shifted Poisson spaces give strong evidence for the statement in Theorem \ref{thm: main}.
Spelling out these constructions sheaf theoretically and taking the appropriate functional-analytic considerations into account are among the problems addressed in this paper.

\subsection{Theorem \ref{thm: main} and Kontsevich's Deformation Quantization}

Let us comment on what Theorem \ref{thm: main} has to say about the deformation quantization of Kontsevich.
Let $V$ be a vector space with a formal Poisson structure, i.e. there is an element $\Pi\in \prod_{n\geq 0} \Sym^n(V^*) \otimes \Lambda^2 V$ such that $[\Pi,\Pi]=0$, where $[\cdot,\cdot]$ is the Schouten-Nijenhuis bracket on (formal) polyvector fields on $V$.
Then, $\widehat{\sO}(V)$, the algebra of power series on $V$, possesses a Poisson bracket of cohomological degree 0.
(The reader may, without missing the point, replace $\widehat{\sO}(V)$ with $\cinfty(P)$, where $P$ is a Poisson manifold. However, the results of this paper and of \autocite{butsonyoo} and \autocite{classicalarxiv} do not apply immediately to the case of general $P$.)
On the other hand, we have discussed observable algebras with Poisson brackets of cohomological degree +1.
In particular, the factorization algebra $\Obcl_V$ of observables for topological Poisson mechanics is a locally constant $\PO$-factorization algebra on $\RR$.
By Theorem 5.5.4.10 of \autocite{higheralgebra}, the $\infty$-category of locally constant factorization algebras on $\RR$ (with values in any ``sufficiently nice'' symmetric-monoidal $\infty$-category $\cC$) is equivalent to the $\infty$-category of algebra objects in $\cC$.
Hence, corresponding to $\Obcl_V$, we obtain a dg associative algebra with values in the symmetric-monoidal category of $\PO$-algebras.
(To be precise, one needs to be careful here, since a $\PO$-factorization algebra is not the same as a factorization algebra whose target symmetric monoidal category is that of $\PO$-algebras; the descent conditions for the two are different. However, since we simply wish to give a schematic argument, we ignore this subtlety for the moment.)
Further, by a result of Safronov \autocite{safronovadditivity}, the $\infty$-category of algebra objects in the symmetric-monoidal $\infty$-category of $\PO$-algebras is equivalent to the $\infty$-category of Poisson algebras.
Hence, $\Obcl_V$ corresponds to a Poisson algebra. 
It is expected, though not known, that this Poisson algebra is equivalent to $\widehat{\sO}(V)$.
If this expectation is indeed a truth, then Theorem \ref{thm: main} implies that the observables of the Poisson sigma model (considered as a $\PO$-factorization algebra) contain all of the information of the Poisson algebra $\widehat{\sO}(V)$. 
This explains the appearance of the Poisson sigma model in the deformation quantization of Poisson manifolds: namely, instead of quantizing the Poisson BV theory on $\RR$ directly, Kontsevich quantized a symplectic BV theory on $\RR\times \RRge$. This symplectic BV theory ``knows all about'' the original Poisson object we set out to quantize.

Theorem \ref{thm: main} may therefore be viewed as justification for the following procedure for the quantization of Poisson BV theories.
First, given a Poisson BV theory $\cT$ on $N$, construct the universal bulk-boundary system on $N\times \RRge$.
Then, study the quantization of this bulk-boundary system using the techniques of \autocite{ERthesis}.
In this way, one can avoid the difficulties posed by the generalization from BV theories on $N$ to Poisson BV theories on $N$.
The consequent trade-off, however, is that one needs to study theories on the manifold $N\times \RRge$, which has a boundary and has larger dimension than $N$ itself.
On the other hand, in studying bulk-boundary systems, one may make use of well-developed machinery for the investigation of boundary-value problems.

\subsection{What Makes the Universal Bulk-Boundary System ``Universal''?}

We conclude with a conjecture which aims to justify the appearance of the word ``universal'' appearing in the term ``universal bulk-boundary system.''
\begin{conjecture}
\label{conj: universalprop}
Let $\cT$ be a PBV theory on $N$, and $\cZ(\cT)$ its universal bulk theory. Let $\rho: N\times \RR_{\geq 0}\to N$ and $\tau: N\times \RR_{\geq 0}\to \RR_{\geq 0}$ denote the respective projections. Then the factorization algebra of classical observables $\Obcl_{\cZ(\cT), \cT}$ for the bulk-boundary system $(\cZ(\cT),\cT)$ is terminal amongst $\PO$-factorization algebras $\cF$ such that 
\begin{enumerate}
    \item there is an equivalence $\Obcl_\cT\to \rho_*\cF$ of $\PO$-factorization algebras on $N$,
    \item for any $U\subseteq N$, the factorization algebra (on $\RRge$) $\tau_*\left( \left.\cF\right|_{U\times \RR_{\geq 0}}\right)$ is stratified locally constant (i.e. inclusions of the form $[0, a) \subseteq [0,a')$ and of the form $(a,b)\subseteq (c,d)$ separately induce equivalences).
\end{enumerate}
In other words, for any $\cF$ satisfying (1) and (2) above, there exists (up to an appropriate notion of homotopy) a map $\cF\to \Obcl_{\cZ(\cT), \cT}$ of factorization algebras on~$N\times \RR_{\geq 0}$.  
\end{conjecture}

Let us marshal some folk intuitions in favor of this conjecture. To this end, it is instructive to consider the case of $E_n$-algebras first. Given an $E_n$-algebra $A$, there exists an $E_{n}$-algebra $\cZ(A)$ with the following properties:
\begin{enumerate}
    \item The $E_n$-algebra $\cZ(A)$ comes equipped with a map $1\to \cZ(A)$.
    \item There is a map $\cZ(A)\otimes A \to A$.
    \item There is a commutativity datum for the diagram
    \[
    \xymatrix{
    & \cZ(A)\otimes A\ar[rd]&\\
    A\ar[ur]\ar[rr]^{\id} & & A
    }.
    \]
    \item The $E_n$-algebra $\cZ(A)$ is final amongst all $E_n$-algebras satisfying the preceding 3 properties.
\end{enumerate}
These four properties are taken in a suitably homotopical sense.
Further, the universal property of $\cZ(A)$ endows it with the structure of an algebra object in the symmetric-monoidal $\infty$-category of $E_n$-algebras. By Dunn-Lurie additivity (\autocite[Theorem 2.9]{dunn}and \autocite[Theorem 5.1.2.2]{higheralgebra}), this implies that $\cZ(A)$ has the structure of an $E_{n+1}$-algebra extending its $E_n$-algebra structure.
More generally, if $A$ is an object in a symmetric monoidal $\infty$-category $\sC$, then an object $\cZ(A)$ satisfying the above four properties is called the \emph{center} of $A$. 
The universal property of $\cZ(A)$ gives it the structure of an algebra object in $\sC$.
For more precise statements, we direct the reader to Section 5.3 of \autocite{higheralgebra}.

We may take, as our symmetric-monoidal $\infty$-category $\sC$, the category of $\PO$-algebras.
In this case, the center is an algebra object in the category of $\PO$-algebras, so by Poisson additivity \autocite{safronovadditivity}, gives a $\POne$-algebra, i.e. an ``ordinary'' Poisson algebra.
There is a---to our knowledge---conjectural description of this $\POne$-algebra in terms of the polyvector fields on the original $\PO$-algebra (see \autocite{safronovpoissreduction}, Remark 1.7.)
This polyvector field description of the center is manifestly a $\POne$-algebra.
Generalizing, we may also choose as $\sC$ the category of $\PO$-factorization algebras on $N$.
Tracing through the constructions, we expect that the center of such a factorization algebra should be a $\POne$-factorization algebra on $N$.
Indeed, Butson and Yoo (implicitly) construct a $\POne$-factorization algebra $\cZ(\Obcl_{\cT})$ on $N$ which assigns to an open subset $U\subseteq N$ a version of the algebra of shifted polyvector fields on the $\PO$-algebra $\Obcl_\cT(U)$. (This factorization algebra is the factorization algebra of observables for the \emph{local higher Poisson center}, Definition 3.17 ibid. This usage of the word ``center'' is suggestive, but---to our knowledge---it is not known whether $\cZ(\Obcl_\cT)$ satisfies the relevant universal property.)

How do we relate the preceding discussion, which happened solely on $N$, to Conjecture \ref{conj: universalprop}, where the factorization algebra $\Obcl_{\cZ(\cT),\cT}$ on $N\times \RRge$ appears?
In this case, it is instructive to return again to the center $\cZ(A)$ of an $E_n$-algebra $A$, and another universal property it possesses. 
As proved in \autocite{Thomas_2016}, $\cZ(A)$ is ``the universal $E_{n+1}$-algebra acting on $A$.'' 
Slightly more precisely, it is the final object in the category of Swiss cheese algebras whose underlying $E_n$-algebra is $A$.
The Swiss Cheese operad was introduced by Voronov \autocite{voronovsc}; a modern understanding of the category of Swiss Cheese algebras is in terms of stratified locally constant factorization algebras on the stratified space $\RR^n\subset \HH^{n+1}$ (where $\HH^{n+1}$ is the upper half-space) (see, e.g. the abstract of \autocite{idrissi2020formality}).
Using this ``modern understanding,'' together with the fact that $E_n$-algebras can be modelled by locally constant factorization algebras on $\RR^n$ \autocite{higheralgebra}, we may rephrase the main theorem of Thomas in \autocite{Thomas_2016} as follows: given a locally-constant factorization algebra $A$ on $\RR^n$, $\cZ(A)$ is terminal amongst locally constant factorization algebras on $\RR^{n+1}$ which appear as the ``bulk algebra'' of a stratified locally constant factorization algebra on $\RR^n\subset \HH^{n+1}$ whose ``boundary'' algebra is $A$.
Conjecture \ref{conj: universalprop} is a translation of this universal property into the realm of $\PO$-factorization algebras.
The factorization algebras $\Obcl_{\cZ(\cT),\cT}$ and $\cZ(\Obcl_\cT)$ are evidently not the same objects on the nose; they live on different spaces. 
However, $\Obcl_{\cZ(\cT),\cT}$ is constructed from the same data as $\cZ(\Obcl_\cT)$. Indeed, if we let $\fX$ denote the space of fields on which $\cZ(\Obcl_\cT)$ is the algebra of functions, then the space of fields for the universal bulk theory $\cZ(\cT)$ is (with some asterisks)~$\underline{\mathrm{Map}}(\RR_{dR}, \fX)$.
Moreover, we expect the two objects to be related in the following way.
Note that for each $U\subseteq N$, the factorization algebra on $\RRge$
\[
\tau_*\left(\left.\Obcl_{\cZ(\cT),\cT}\right\vert_{U\times \RRge}\right)
\]
(notation from Conjecture \ref{conj: universalprop}) is stratified locally constant, so in particular, it gives a locally constant factorization algebra on $\RRgt$.
Now, consider the equivalences of $\infty$-categories
\[
\POne\text{-}\mathrm{Alg} \simeq \mathrm{AsAlg}(\PO\text{-}\mathrm{Alg}) \simeq \mathrm{FA}^{l.c.}_{\RR_{>0}}(\PO\text{-}\mathrm{Alg}),
\]
where the first equivalence is the main theorem of \autocite{safronovadditivity} and the second result is proved by Lurie in \autocite{higheralgebra}.
We expect that under this equivalence, the objects $\cZ(\Obcl_\cT)(U)$ and $\tau_*\left(\left.\Obcl_{\cZ(\cT),\cT}\right|_{U\times \RRgt}\right)$ coincide, and give different models for the two different universal properties of the center.
We also expect a compatibility between these identifications as $U$ varies over the open subsets of $N$; however, the nature of this compatibility isn't entirely clear to the author at this moment.

\subsection{Plan of the Paper}
The paper is organized as follows.
In Section \ref{sec: PODisj}, we set notations related to the theory of (colored) operads (Section \ref{subsec: operadconventions}), and we describe the colored operad governing $\PO$-prefactorization algebras (Section \ref{subsec: PODisj}).

In Section \ref{sec: PBV}, we briefly describe the background necessary to understand the objects appearing in Theorem \ref{thm: main}, including background on Poisson BV theories and the $\PO$-factorization algebras associated to a Poisson BV theory. 
We give a few examples of these objects, and apply Theorem \ref{thm: main} to those examples.

In Section \ref{sec: proof}, we execute the proof of Theorem \ref{thm: main}. The main tools of the proof are the use of the homological perturbation lemma to construct open-by-open deformation retractions
\[
\xymatrix{
\Obcl_\cT(U) \ar@<1ex>[r]^-{\iota(U)} & \Obcl_{\cZ(\cT)}(U\times \RRge)\ar@<1ex>[l]^-{\pi(U)}{\ar@(ur,dr)[]+R+<0pt,4pt>;{[]+R+<0pt,-4pt>}^{\eta(U)}}
}
\]
and the homotopy transfer of $\hoPODisj$-algebra structures, as made explicit in Section \ref{subsec: operadconventions}.

\subsection{Conventions}
\label{sec: conventions}
\begin{itemize}
    \item The symbol $\mathbb K$ denotes either of the fields $\RR$ or $\CC$.
    \item $\PO$ is the operad whose algebras are Poisson algebras with Poisson bracket of cohomological degree +1. For precise signs, refer to Definition 1.1. of \autocite{safronovpoissreduction} 
    \item If $V$ and $W$ are cochain complexes, then a \emph{strong deformation retraction of $V$ onto $W$} is a triple $(\iota,\pi,\eta)$, where 
    \[
    \pi: V\leftrightarrow W:\iota
    \]
    are cochain maps satisfying $\pi\iota = \id_W$; $\eta$ is a cochain homotopy witnessing the exactness of $\id_V-\iota\pi$; and the three maps satisfy the ``side conditions''
    \[
    \eta^2 = \pi \eta = \eta\iota = 0,
    \]
    which simplify many computations.
    \item If we use a normal-font Latin letter (e.g. $L$) to denote a bundle on a manifold $N$, then we use the corresponding script-font letter (e.g. $\sL$) to denote the corresponding sheaf of sections on $N$. We will occasionally also use the notation $\cinfty(U; L)$ for the sections of $L$ over $U$.
    \item If $V$ is a finite-dimensional vector space, we denote by $V^*$ its dual. Similarly, if $V\to N$ is a vector bundle, we denote by $V^*$ its fiberwise dual.
    \item We denote by $\Dens_N$ the bundle of densities on $N$. Its sheaf of sections is denoted $\Denssheaf_N$.
    \item If $V$ is a vector space and $N$ is a manifold, we use $\underline{V}$ to denote the trivial bundle on $N$ with fiber $V$.
    \item If $V\to N$ is a vector bundle, $V^!$ denotes the bundle $V^*\otimes \Dens_N$.
    \item We denote by $\Lambda^\sharp T^* N$ the ($\ZZ$-graded) exterior algebra bundle on the cotangent bundle of $N$. Its sheaf of sections is $\Omega^\sharp_N$. For the special case $N=\RRge$, which appears prominently in this article, we write simply $\Omega^\sharp$. $\Omega^\bullet_N$ denotes the de Rham complex of sheaves on $N$. 
    Finally, we write $\Omega^\sharp_D$ for the sheaf on $\RRge$ consisting of forms whose pullback to the boundary point vanishes (the ``D'' stands for Dirichlet).
    \item If $V\to N$ is a bundle, then we let $\sV_c$ denote the cosheaf of compactly-supported sections of $V$. Following this notation, $\Omega^\sharp_{c,D}$ is the precosheaf of compactly-supported forms on $\RRge$ whose pullback to the boundary point vanishes.
    \item If $V\to N$ is a bundle, then $V_x$ denotes the fiber of $V$ over $x\in N$.
    \item If $V_1$ is a bundle over $N_1$ and $V_2$ is a bundle over $N_2$, then $V_1\boxtimes V_2$ is the bundle over $N_1\times N_2$ whose fiber at $(x,y)$ is $(V_1)_x\otimes (V_2)_y$.
    \item We denote by $\mathbb{S}_k$ the permutation group on $k$ letters. If $V$ is an $\mathbb{S}_k$-module, then $V_{\mathbb{S}_k}$ denotes the space of coinvariants for the action, and $V^{\mathbb{S}_k}$ the space of invariants.
    \item We use cohomological $\ZZ$-grading conventions, i.e. all differentials have cohomological degree $+1$, and if $V$ is a complex, $V[1]$ denotes the same complex shifted \emph{down} by one.
    \item In keeping with the notation $\Omega^\sharp$, given a cochain complex $B$, we will let $B^\sharp$ denote the underlying graded object.
    \item We often deal with a collection of objects $A(U)$, one for each open subset $U\subseteq N$. We will sometimes omit the $(U)$ from the notation when this is clear from context.
    \item When we say ``tree,'' we mean ``rooted tree,'' which is described by a set $V(T)$ of vertices, a set $HE(T)$ of half-edges, a map $HE(T)\to V(T)$, an involution $\sigma: HE(T)\to HE(T)$, and a distinguished element $r\in HE(T)^{\mathbb{S}_2}$. We refer to $r$ as the \emph{root} of $T$, the other fixed points of $\sigma$ as the \emph{leaves} of $T$, and the $\sigma$-orbits of size 2 as the \emph{edges} or \emph{internal edges} of $T$.
\end{itemize}

\subsection{Conventions Concerning Functional Analysis}
In this section, we discuss the functional analytic context in which we work.
In this article, we require very little actual functional analysis, but we do rely on some general categorical properties of our underlying category of vector spaces.
The goal of this section is to present those properties in as brief a manner as possible.
To put it differently, the goal of this section is to convince you that this section is almost unnecessary.

Whenever the word ``(co)chain complex'' appears in the body of the text, it is to be understood to mean ``(co)chain complex of differentiable vector spaces over $\RR$.'' 
A differentiable vector space is in particular a vector-space-valued sheaf on the site of smooth manifolds.
For example, given a manifold $N$, one obtains the following sheaf
\[
X\mapsto \cinfty(N\times X),
\]
which is meant to encode the topological vector space $\cinfty(N)$.
The category $\DVS$ of differentiable vector spaces has the following properties:
\begin{itemize}
    \item It is abelian.
    \item It is enriched over itself.
    \item Given a manifold $N$, the vector space $\cinfty(N)$ is naturally a differentiable vector space.
    \item It has a multicategory structure such that 
    \[
    \Hom_\DVS(\cinfty(N_1), \cinfty(N_2); V)\cong \Hom_\DVS(\cinfty(N_1\times N_2);V),
    \]
    i.e. the multicategory structure resembles the completed projective tensor product of locally convex topological vector spaces.
\end{itemize}
See Appendix B of \autocite{CG1} for more concrete details.
Suffice it to note that it forms a suitable context for homological algebra with vector spaces of smooth functions on manifolds.

Though the category $\DVS$ has many nice properties, a differentiable vector space can be an unwieldy object, since one needs to remember what it assigns to every manifold.
In practice, we will often construct differentiable vector spaces out of \emph{convenient vector spaces} (see again Appendix B of \autocite{CG1} for the details necessary for this article or \autocite{krieglmichor} for the original reference).
Convenient vector spaces are more concrete than differentiable vector spaces: they arise as topological vector spaces satisfying an additional completeness property.
The category $\CVS$ is a full subcategory of $\DVS$.
It is closed symmetric monoidal. We denote by $\hotimes_\beta$ its symmetric monoidal product.
We use the notation $\Hom_\CVS(\cdot, \cdot)$ to denote the internal hom bifunctor (we have no use for the unenriched hom functor, so we do not fuss about having notation which distinguishes the two).
The fully-faithful embedding $\CVS\to \DVS$ has the following properties:
\begin{itemize}
    \item It preserves limits.
    \item It preserves inner homs.
    \item The multicategory structure is represented on $\CVS$ by $\hotimes_\beta$, i.e. for any convenient vector spaces $W_1,W_2, V$, we have that~$\Hom_\DVS(W_1,W_2;V) \cong \Hom_\DVS(W_1\hotimes_\beta W_2; V)$.
\end{itemize}
The functor $\CVS\to \DVS$ \emph{does not} preserve colimits.
It is therefore usually important to compute cohomology or construct quasi-isomorphisms in $\DVS$ and not $\CVS$. In this article, we will always construct deformation retracts, so this subtlety will not appear.

We mention one last reason that functional-analytic subtleties do not appear: the equivalence of Theorem \ref{thm: main} more or less amounts to ``integrating out'' the normal direction $\RRge$.
To accomplish this integration, the analytical tools involved are no more complicated than those involved in constructing a deformation retraction of the de Rham complex $\Omega^\bullet_\RR$ onto its cohomology.

\subsection{Acknowledgements}
The author would like to warmly thank Dylan Butson, Damien Calaque, Owen Gwilliam, Najib Idrissi, Nima Moshayedi, Pavel Safronov, Stephan Stolz, Peter Teichner, Minghao Wang, Brian Williams, and Philsang Yoo for conversations related to the material presented here.

He would like to thank V\'ictor Carmona and Thomas Willwacher for separately suggesting the use of the bar-cobar resolution of $\PO\text{-}\Disj_N$.

\section{Background on  $\PO$-Factorization Algebras and their Homotopy Theory}
\label{sec: PODisj}
This section and the next are devoted to a review of relevant background material.
First, in this section, we discuss the general homotopy theory of $\PO$-prefactorization algebras from the standpoint of operad theory.
More precisely: in Section \ref{subsec: operadconventions}, we establish notation and recall various important facts concerning the theory of operads. In Section \ref{subsec: PODisj}, we apply that general theory of operads to the colored operad $\PO\text{-}\Disj_N$ whose algebras are $\PO$-prefactorization algebras.

Then, in Section \ref{sec: PBV}, we will construct the objects appearing in Theorem \ref{thm: main}.

\subsection{Fixing Notation for Operads and their Algebras}
\label{subsec: operadconventions}
We refer the reader to the reference \autocite{LodayVallette} for the general theory of operads which we apply here.
In particular, in this subsection, we make absolutely no claims about originality.
We use the following notation for general operads and cooperads, 
which follows that of \autocite{LodayVallette} with one exception: we use cohomological instead of homological grading conventions.
We will tacitly assume all of our operads to be colored by a set $S$, whether or not this is mentioned explicitly in the terminology and notation.
Given an augmented dg-operad $\sP$ and a conilpotent dg-cooperad $\sC$, we let:
\begin{itemize}
    \item If $\sP$ is an $\mathbb{S}$-module, then $\cT(\sP)$ and $\cT^c(\sP)$ are respectively the free operad and cooperad on $\sP$. As $\bbS$-modules, they are the same: the elements of arity $k$ are formal linear combinations of rooted trees $T$ with a bijection from the set of leaves of $T$ to $\{1,\ldots,k\}$; vertices of $T$ with $\ell$ inputs are labeled by $\ell$-ary elements of $\sP$.
    \item $\mathrm {B}\sP$ denote the bar construction on $\sP$, which gives a colored conilpotent dg-cooperad,
    \item $\Omega \sC$ denote the cobar construction on $\sC$, which gives an augmented colored dg-operad,
    \item If $\sP$ is an operad, $\mu$ is an operation in $\sP$, and $A$ is a $\sP$-algebra, then we let $\mu_A$ denote the representation of $\mu$ in $A$.
\end{itemize}

Given an augmented operad $\sP$, there is always a quasi-isomorphism
\[
\Omega \mathrm{B} \sP \overset{\sim}{\to} \sP
\]
of operads.
The operad $\Omega \mathrm{B}\sP$ has in particular an operation $\tilde \mu$ for any operation $\mu$ in $\sP$, but $\tilde \mu \circ_i \tilde \nu \neq \widetilde{\mu\circ_i \nu}$.
Instead there is another operation in $\Omega\mathrm{B} \sP$ exhibiting a homotopy between the two operations.
There are further operations witnessing homotopies between homotopies, and so on.
In fact, the operad $\Omega \mathrm{B}\sP$ has a generating operation for every \emph{tree} whose vertices are labeled by elements of $\sP$ of the appropriate arity. The operation $\tilde \mu$ corresponds to the one-vertex tree with vertex labeled by $\mu$.

In the model category of operads introduced by Hinich \autocite{hinichoperads}, the operad $\Omega \mathrm{B}\sP$ is a cofibrant replacement for $\sP$.
One of the nice things about resolutions of an operad $\sP$ of the form $\Omega\sC$, where $\sC$ is a conilpotent operad, is that---although the notion of weak equivalence of $\sP$-algebras is not an equivalence relation---there is a notion of weak-equivalence of $\Omega \sC$-algebras which is an equivalence relation and which coincides with the equivalence relation generated by weak equivalence for $\sP$-algebras.
Indeed, one may define an $\infty$-morphism $A\rightsquigarrow A'$ of $\Omega B \sP$-algebras to be a morphism of the cofree coalgebras
\[
\mathrm B \sP(A)\to \mathrm B \sP(A'),
\]
where both objects are endowed with the unique square-zero coderivation induced from the $\Omega\mathrm{B}\sP$-algebra structures.

$\infty$-morphisms are more flexible than strict morphisms of operadic algebras, and in addition, one can find an $\infty$-quasi-isomorphism of $\Omega\mathrm{B}\sP$ algebras if and only if there is a zig-zag of quasi-isomorphisms of strict $\Omega\mathrm{B}\sP$-algebras.

Moreover, as we will see in a moment, given a $\Omega\mathrm{B}\sP$-algebra $A$, its cohomology as a colored cochain complex $H^\bullet(A)$ has an $\Omega \mathrm{B}\sP$-algebra structure such that there is an $\infty$-quasi-isomorphism $H^\bullet(A)\rightsquigarrow A$.

More precisely, there is a homotopy transfer theorem; in Loday-Vallette, the arguments are made for a subcooperad $\sP^{\coshriek}$ of $B \sP$ which one may define in certain cases, but as noted in \autocite[Remark 2.7.2.(i)]{merkulov}, the same arguments apply also for $B\sP$.
\begin{theorem}[cf. Theorem 10.3.1 of \autocite{LodayVallette}]
Let $A'$ be a $\sP$-algebra (where the set of colors of $\sP$ is $S$) and consider a collection of deformation retractions
\[
\xymatrix{
A(s) \ar@<1ex>[r]^-{\iota(s)} & A'(s)\ar@<1ex>[l]^-{\pi(s)}{\ar@(ur,dr)[]+R+<0pt,4pt>;{[]+R+<0pt,-4pt>}^{\eta(s)}}
},
\]
one for each $s\in S$, such that 
\[
\pi(s)\circ \iota(s) = \id_{A(s)},\quad [d_{A'(s)},\eta(s)]=\id_{A'(s)}-\iota(s)\circ\pi(s).
\]
Then there is a $\Omega\mathrm B \sP$-algebra structure on $A'$ and an $\infty$-quasi-isomorphism $A\rightsquigarrow A'$ extending $\iota$.
\end{theorem}

The theorem applies even if $A'$ is only a $\Omega\mathrm B \sP$-algebra, but we will not need that fact here.
However, it will be necessary below to understand the structure induced on $A$ by the above theorem in the case that $A'$ is a $\sP$-algebra.
Since the operad $\Omega \mathrm B \sP$ is semi-free on $\mathrm B \sP[-1]$, the transferred structure on $A'$ is determined by a collection of operations, one for each cooperation in $\mathrm B \sP$.
The cooperad $\mathrm B \sP$ is itself semi-cofree on $\overline{\sP}[1]$.
Cooperations in $\mathrm B \sP$ correspond to nonplanar rooted trees labelled by elements of $\overline{\sP}$.
Given such a tree $T$, one obtains a multilinear operation on $A$ in the following way.
Associate to the ``leaves'' of $T$ the map $\iota$, to a $\mu$-labeled vertex of $T$ its representation $\mu_{A'}$, and to the root of $T$ the operation $\pi$.
In this way, one obtains a multilinear operation on $A$ which describes the representation of $T$ in $A$.
Figure \ref{fig: exampleHTT} gives an example of such a tree.
The infinity-quasi-isomorphism extending $\iota$ is described similarly, except the root of a tree is labelled instead by $\eta$.

\begin{figure}
    \centering
\begin{tikzpicture}[scale=1.5]
         \node at (-.25,0){$\mu_{A'}$};
         \draw[ultra thick] (0, 0) arc (0:60:.25cm) --+(0,0) node(a){};
         \draw[ultra thick] (a) arc (60:120:.25cm) --+(0,0) node(b){};
        \draw[ultra thick] (b) arc (120:180:.25cm) --+(0,0) node(c){};
        \draw[ultra thick] (c) arc (180:300:.25cm) --+(0,0) node(d){};
        \draw[ultra thick] (d) arc (300:360:.25cm) --+(0,0) node{};
        \draw (a.center) -- +(60:.5) node[above]{$\iota$};
        \draw (b.center) -- +(120:.5) node[above]{$\iota$};
        \draw (d.center) -- +(-60:.5) node[pos=.65,left]{$\eta$} node (g){};
        \draw[ultra thick] (g.center) arc (120:270:.25cm)--+(0,0) node(h){};
        \draw[ultra thick] (h.center) arc (270:360:.25cm) --+(0,0) node(i){};
        \draw[ultra thick] (i.center) arc (0:60: .25cm) --+(0,0) node(j){};
        \draw[ultra thick] (j.center) arc (60:90: .25cm)--+(0,0) node(k){};
        \draw[ultra thick] (k.center) arc (90:120: .25cm)--+(0,0) node(l){};
        \path (k.center) -- +(0,-.25) node {$\nu_{A'}$};
        \draw (h.center) --+(-90: .5) node[below]{$\pi$};
        \draw (j.center) --+(60: .5) node[above]{$\iota$};
        \draw (k.center) --+(90:.5) node[above]{$\iota$};
         \end{tikzpicture}
         \label{fig: eep}
\caption{A diagrammatic example of transferred operations.}
\label{fig: exampleHTT}
\end{figure}
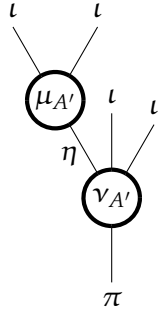

\subsection{$\PO$-Prefactorization Algebras}
\label{subsec: PODisj}
In this subsection, we describe the colored operad to which we apply the general theory of the preceding subsection.
This is the colored operad governing $\PO$-prefactorization algebras on a manifold $N$.

\begin{definition}
Let $N$ be a manifold. Let $\Disj_N$ be the colored operad whose colors are the open subsets of $N$ and whose operations are given by the formula
\[
\Disj_N(U_1,\ldots, U_k; V) = \left\{
\begin{array}{lr}
     \mathbb K & \text{ if the } U_i \text{ are pairwise disjoint subsets of }V \\
     0& \text{otherwise} 
\end{array}
\right.
\]
A \textbf{prefactorization algebra on $N$} is an algebra (in the category of chain complexes) over this colored operad.
\end{definition}

In more detail, a prefactorization algebra $\cF$ assigns a cochain complex $\cF(U)$ to each open subset $U\subseteq N$, and a cochain map
\[
m_{U_1,\ldots, U_k}^V : \cF(U_1)\otimes \cdots \otimes \cF(U_k) \to \cF(V)
\]
whenever the $U_i$ are pairwise disjoint subsets of $V$.
Furthermore, given any permutation $\sigma \in \mathbb{S}_k$, we require the following diagram to commute:
\[
\xymatrix{
\cF(U_1)\otimes \cdots \otimes \cF(U_k) \ar[r]^-{m_{U_1,\ldots, U_k}^V}\ar[d]^{\sigma^{-1}} & \cF(V)\\
\cF(U_{\sigma(1)})\otimes \cdots \otimes \cF(U_{\sigma(k)})\ar[ur]_-{m_{U_{\sigma(1)},\ldots, U_{\sigma(k)}}^V}&
}
\]
Finally, if $U_{i1},\ldots, U_{ik_i}$ are pairwise disjoint subsets of $V_i$ for $1\leq i\leq n$ and the $V_i$ are pairwise disjoint subsets of $W$, we require the following diagram to commute:
\[
\xymatrix{
\cF(V_1)\otimes \cdots \cF(V_n)\ar[r]& \cF(W)\\
\cF(U_{11})\otimes \cdots \otimes \cF(U_{nk_n})\ar[u]\ar[ur]&
}
\]

What makes a $\PO$-prefactorization algebra, as opposed to a ``regular'' prefactorization algebra? The operad $\PO$ is Hopf, so it makes sense to define the (symmetric monoidal) tensor product of $\PO$-algebras.
Recall also that, given a colored operad $\sP$, it makes sense to ask for a representation of/algebra over $\sP$ in any $\mathbb k$-linear symmetric monoidal category.

\begin{definition}
A \textbf{$\PO$-prefactorization algebra} is a representation of $\Disj_N$ in the symmetric monoidal category of $\PO$-algebras.
\end{definition}

The preceding definition is a compact description of $\PO$-prefactorization algebras, but it will be necessary for us to also present a colored operad whose algebras in the symmetric monoidal category of \emph{vector spaces} are $\PO$-prefactorization algebras:

\begin{definition}
\label{def: PODisj}
Let $N$ be a manifold. The $\mathrm{Opens}_N$-colored operad $\PO\text{-}\Disj_N$ is generated by
\begin{enumerate}
    \item the $k$-ary operations $m_{U_1,\ldots, U_k}^V$, of color $\binom{V}{U_1,\ldots, U_k}$, defined for every pairwise disjoint collection $\{U_i\}$ of open subsets of $V$.
    \item the binary operations $\mu_U$, of color $\binom{U}{U,U}$, defined for each open subset $U$ of $N$.
    \item the binary operations $\varpi_U$, of color $\binom{U}{U,U}$, defined for each open subset $U$ of $N$.
\end{enumerate}
These operations are subject to the relations:
\begin{enumerate}
    \item The $m_{U_1,\ldots, U_k}^V$ satisfy the relations of $\Disj_N$.
    \item For each $U$, the operations $\varpi_U$ and $\mu_U$ satisfy the relations that, respectively, the Poisson bracket and commutative product do in $\PO$.
    \item The maps $m_{U_1,\ldots, U_k}^V$ respect the $\PO$-structures.
\end{enumerate}
\end{definition}

It is striaghtforward to show that an algebra over $\PO\text{-}\Disj_N$ in the category of cochain complexes is the same as a $\PO$-prefactorization algebra.
By a \emph{morphism} or \emph{quasi-isomorphism} of ($\PO$-)prefactorization algebras, we shall mean a morphism of algebras over the colored operad $(\PO)\text{-}\Disj_N$.

The following lemma is immediate:
\begin{lemma}
\label{lem: PODisjgens}
The operad $\Disj_N$ is generated by the operations $m_{U,V}^{U\sqcup V}$, $m_U^V$.

The operad $\PO\text{-}\Disj_N$ is generated by the operations $m_{U,V}^{U\sqcup V}$, $m_U^V$, $\mu_U$, and $\varpi_U$.
\end{lemma}

In particular, to describe a $\Disj_N$ algebra, one only needs to describe the unary operations and a subset of the binary operations.

The operad $\PO\text{-}\Disj_N$ is augmented, so one obtains the resolution
\[
\hoPODisj:= \Omega\mathrm{B}\left(\PO\text{-}\Disj_N\right)\to \PO\text{-}\Disj_N.
\]
From these two operads, one may obtain two relative categories: the category of $\PO$-prefactorization algebras on $N$ whose weak equivalences are quasi-isomorphisms, and the category of $\hoPODisj$-algebras whose morphisms and weak equivalences are $\infty$-morphisms and $\infty$-quasi-isomorphisms, respectively.
The general theory of operads assures that the $\infty$-categories obtained from these relative categories are equivalent.
Hence Theorem \ref{thm: main} states that there is an equivalence of objects in the $\infty$-category of $\PO$-prefactorization algebras on $N$.

\begin{definition}
Let $\cF$ be a ($\PO$-)prefactorization algebra on $N$ and $f: N\to M$ a continuous map of topological spaces. Then, the \textbf{pushforward} of $\cF$ via $f$ is the ($\PO$-)prefactorization algebra which assigns, to each open subset $U\subset M$, the space
\[
\rho_*(\cF)(U) := \cF(f^{-1}(U));
\]
this definition is completely analogous to the pushforward operation for (co)sheaves.
\end{definition}

\begin{remark}
There is a codescent condition one may formulate for ($\PO$-)prefactorization algebras \autocite{CG1}.
($\PO$-)prefactorization algebras satisfying this additional condition are known as ($\PO$-)factorization algebras.
Factorization algebras form a full subcategory of prefactorization algebras,
and the two main $\PO$-prefactorization algebras appearing in Theorem \ref{thm: main} are in fact $\PO$-factorization algebras.
Theorem \ref{thm: main} implies that there is a zigzag of quasi-isomorphisms of $\PO$-prefactorization algebras between $\Obcl_\cT$ and $\Obcl_{\cZ(\cT),\cT}$.
Since both of these objects are factorization algebras, and since the intermediate $\PO$-prefactorization algebras are related by quasi-isomorphism to each other, it follows that all of the intermediate $\PO$-prefactorization algebras are indeed $\PO$-factorization algebras.
Since factorization algebras form a full subcategory of prefactorization algebras, it follows that there is a zigzag of quasi-isomorphisms of $\PO$-factorization algebras between $\Obcl_{\cT}$ and $\Obcl_{\cZ(\cT),\cT}$.
\end{remark}

\section{The $\PO$-Factorization Algebras for Poisson BV Theories and their Universal Bulk-Boundary Systems}
\label{sec: PBV}
In this section, we summarize definitions, constructions, and examples related to Poisson BV theories and bulk-boundary systems.
These definitions have appeared in \autocite{butsonyoo} and \autocite{classicalarxiv}.
Hence, we will not give full definitions here; mostly, we will describe the data appearing in the definitions, omitting the conditions those data satisfy.
Recall from the introduction that, given a Poisson BV theory $\cT$, we obtain $\PO$-factorization algebras $\Obcl_\cT$ and $\Obcl_{\cZ(\cT),\cT}$ on $N$ and $N\times \RRge$, respectively.
Hence, we need to define Poisson BV theories and the associated $\PO$-prefactorization algebras $\Obcl_\cT$, and $\Obcl_{\cZ(\cT),\cT}$.
We start with the definition of a Poisson BV theory.

\begin{definition}[Cf. Definition 2.30 of \autocite{butsonyoo}]
Let $N$ be a manifold (without boundary). A \textbf{Poisson BV theory} on $N$ consists of 
\begin{enumerate}
    \item A $\ZZ$-graded real or complex vector bundle $L\to N$.
    \item For each $k\geq 1$, a $k$-ary polydifferential operator
    \[
     \ell_{k,\sL}: \sL^k \to \sL
    \]
    of cohomological degree $+1$.
    \item For each $r\geq 0$, a polydifferential operator
    \[
    \Pi^{(r)}: \sL^r \times \sL^! \to \sL
    \]
    of cohomological degree $+1$.
\end{enumerate}
These data are required to satisfy a number of compatibilities and symmetry properties, which have geometric interpretations. In particular, the operations $\ell_{k,\sL}$ turn $\sL[-1]$ into a presheaf of $L_\infty$-algebras on $N$ (so one may think of $\sL$ as a formal moduli problem or $\sL[1]$ as a $Q$-manifold), the $\Pi^{(r)}$ describe the Taylor components of the natural map $T^*\sL \to T\sL$ induced by a Poisson bivector, and the Lie derivative of this Poisson bivector along the cohomological vector field is zero.
We will often use the notation $\cT$ as shorthand for the entire collection $(L, \ell_{k,\sL}, \Pi^{(r)})$ appearing in this definition.
\end{definition}

Let us mention a few examples here, just to give the reader a flavor for the possibilities.
We refer the reader to \autocite{butsonyoo} for more examples.
\begin{example}
Let $V$ be a vector space and $\omega = \sum_{r\geq 0} \omega^{(r)}$ be a formal Poisson bivector on $V$, i.e. $\omega^{(r)}$ is a linear map $\Sym^r(V)\otimes V^* \to V$, and the sum $\omega$ satisfies $[\omega,\omega]=0$, where $[\cdot,\cdot]$ is the Schouten-Nijenhuis bracket of polyvectors.
Set $N= \RR$, $L = (\Lambda^\sharp T^* \RR)\otimes V$, define the operations $\ell_{k,\sL}$ to be trivial for $k>1$, and for $k=1$, set
\[
\ell_{1,\sL} =d_{dR}: \Omega^\sharp_{\RR} \otimes V\to \Omega^\sharp_\RR \otimes V.
\]
In this case, $\sL^! = \Omega^\sharp_\RR[1]\otimes V^*$, and we define---for $\alpha_1,\ldots, \alpha_{r+1} \in \Omega^\sharp_\RR$, $v_1,\ldots, v_{r}\in V$, and $\nu\in V^*$---the Taylor components $\Pi^{(r)}$ by the formula
\[
\Pi^{(r)}( \alpha_1\otimes v_1,\ldots, \alpha_r\otimes v_r, \alpha_{r+1}\otimes \nu) = \left(\alpha_1\wedge \alpha_2\wedge \cdots \wedge \alpha_r\right) \otimes \omega^{(r)}(v_1,\ldots, v_r, \nu).
\]
The cohomological degree of $\Pi=\sum_r \Pi^{(r)}$ is +1 on account of the shift in the equation $\sL^! = \Omega^\sharp_\RR[1]\otimes V^*$.
\end{example}

\begin{example}
\label{ex: cswzw}
Let $\Sigma$ be a Riemann surface and $\fg$ a Lie algebra with an invariant pairing $\kappa$.
Set $L= \Lambda^\sharp (T^*_\CC)^{(1,0)}\Sigma \otimes \fg^*$ (concentrated in cohomological degrees 0 and 1). 
Just as in the previous example, set $\ell_{k,\sL}= 0$ for $k>1$; set $\ell_1$ to be the Dolbeault differential on the space of $(1,\sharp)$ forms.
In this case, $\sL^! = \Omega^{0,\sharp}_\Sigma \otimes \fg[1]$, and we define
\[
\Pi^{(0)}(\alpha\otimes x) = \del\alpha \otimes \kappa(x,\cdot)
\]
For $\beta\in \Omega^{1,\sharp}_\Sigma$, $\alpha\in \Omega^{0,\sharp}_\Sigma$, $x\in \fg$, $\chi \in \fg^*$, we also define
\[
\Pi^{(1)}(\beta\otimes \chi, \alpha \otimes x) = -\beta\wedge \alpha \otimes [x,\chi],
\]
where the bracket $[x,\chi]$ denotes the coadjoint action of $x$ on $\chi$.
Define all other $\Pi^{(r)}$ to be zero.
For $\fg$ abelian, this is the central example of \autocite{GRW}.
\end{example}

\begin{example}
\label{ex: scalarft}
Let $(N,g)$ be a Riemannian manifold, and set $L = \underline{\RR}\oplus \underline{\RR}[-1]$.
Define 
\[
\ell_{1,\sL} : \cinfty_N\to \cinfty_N
\]
to be the Laplace-Beltrami operator; set all other $\ell_{k,\sL}$ to 0.
In this case, $\sL^! = \Denssheaf_N[1]\oplus \Denssheaf$. The Riemannian volume form on $N$ gives an isomorphism $\Denssheaf\cong \cinfty_N$, and we set $\Pi^{(0)}$ to be the natural extension of this isomorphism.
The reader may note that this is a symplectic BV theory; this example serves to demonstrate that symplectic BV theories are also Poisson BV theories.
In particular, they are the Poisson BV theories where $\Pi^{(0)}$ is induced from a bundle isomorphism and $\Pi^{(r)}=0$ for $r>0$.
\end{example}

Given a Poisson BV theory $\cT$ on $N$, we would like to define the $\PO$-prefactorization algebra $\Obcl_{\cT}$ on $N$.
Heuristically, it should assign to each open subset $U\subseteq N$ the algebra of functions on the formal space $\sL(U)$, together with the Chevalley-Eilenberg differential induced from $L_\infty$-algebra structure on $\sL(U)[-1]$.
To describe the Poisson bracket, however, we need to address some functional analytic subtleties first.

\begin{deflem}[cf. Theorem 6.6.1 of \autocite{CG1}]
Suppose that $\cT$ is a Poisson BV theory on $N$. Define, for any open $U\subseteq N$,
\[
C^\bullet_\cT(U) = \left( \prod_{k\geq 0} \Hom_\CVS(\sL(U)^{\hotimes_\beta k},\RR)_{\bbS_k},d_{CE}\right),
\]
where $d_{CE}$ is the Chevalley-Eilenberg differential arising from the $L_\infty$-brackets on $\sL(U)[-1]$.
For $U_1,\ldots, U_k$ pairwise disjoint, define the composites
\[
m_{U_1,\ldots, U_k}^V: C^\bullet_\cT(U_1)\hotimes_\beta \cdots \hotimes_\beta C^\bullet_\cT(U_k) \to C^\bullet_\cT(V)\hotimes_\beta \cdots \hotimes_\beta C^\bullet_\cT(V) \to C^\bullet_\cT(V),
\]
where the first map is a tensor product of extension-by-zero maps and the second is multiplication in the dg commutative algebra $C^\bullet_\cT(V)$.
These composites respect the dg commutative algebra structures on source and target, and define a factorization algebra structure on $C^\bullet_\cT$, which we call the \textbf{factorization algebra of fully distributional classical observables}.
\end{deflem}

The fully distributional classical observables are easier to define and more intuitive than the object $\Obcl_{\cT}$ which appears in Theorem \ref{thm: main}, but the na\"ive thing one might do to define the $\PO$-bracket of fully distributional classical observables does not work because of functional-analytic issues. 
To this end, we have to introduce the notion of functionals with smooth derivative:
\begin{definition}
\label{def: smoothB}
A \textbf{functional with smooth derivative} with support on $U$ and order $k$ is an element
\[
\varphi \in \Hom_\CVS(\sL(U)^{\hotimes_\beta k}, \RR)_{\bbS_k}
\]
all of whose $\bbS_k$ representatives, when understood as maps (obtained from currying in the last tensor factor)
\[
\sL(U)^{\hotimes_\beta (k-1)} \to \Hom_\CVS(\sL(U),\RR),
\]
factor through the inclusion
\[
\sL^!_c(U)\to \Hom_\CVS(\sL(U),\RR)
\]
of smooth, compactly-supported distributions into the space of all compactly-supported distributions.
We denote by $B_{(k)}(U)$ the space of functionals with smooth derivative with support on $U$ and order $k$.
\end{definition}

Alternatively, choosing a representative in $\Hom_\CVS(\sL(U)^{\hotimes_\beta k},\RR)$ for $\varphi$, we obtain a collection of $k$ maps
\[
\delta_i\varphi \in \Hom_\CVS(\sL(U)^{\hotimes_\beta(k-1)}, \Hom_\CVS(\sL(U),\RR)).
\]
Although each $\delta_i\varphi$ depends on the choice of representative, the set $\{\delta_1\varphi,\ldots, \delta_k \varphi\}$, considered as a collection of elements of
\[
\Hom_\CVS(\sL(U)^{\hotimes_\beta(k-1)}, \Hom_\CVS(\sL(U),\RR))_{\bbS_{k-1}},
\]
is independent of this choice.
Then, $\varphi$ has smooth first derivative if and only if each $\delta_i\varphi$ lifts to a map
\[
\Hom_\CVS(\sL(U)^{\hotimes_\beta(k-1)}, \sL^!_c(U)).
\]
Moreover, we obtain a well-defined element (independent of the choice of representative for $\varphi$)
\[
\delta\varphi = \sum_i \delta_i\varphi \in \Hom_\CVS(\sL(U)^{\hotimes_\beta(k-1)}, \sL^!_c(U))_{\bbS_{k-1}}.
\]
The symbol should be understood as the de Rham differential applied on the space of functionals on $\sL$.

\begin{deflem}[cf. Section 5.4 of \autocite{CG2} and Theorem 2.34 of \autocite{butsonyoo}]
\label{deflem: BPOstruct}
Define
\[
\Obcl_{\cT}(U):= \prod_{k\geq 0} B_{(k)}(U);
\]
then $\Obcl_{\cT}$ is a subfactorization algebra of $C^\bullet_\cT$.
Moreover, the dg commutative algebra structure on $\Obcl_{\cT}(U)$ can be extended to a $\PO$-algebra structure such that $\Obcl_\cT$ forms a $\PO$-prefactorization algebra. If the complex $(\sL,\ell_{1,\sL})$ is elliptic, then the inclusion $\Obcl_\cT\to C^\bullet_\cT$ is a quasi-isomorphism.
\end{deflem}
\begin{proof}
As noted in \autocite{butsonyoo}, the proof is given by adapting the proof given in Section 5.4 of \autocite{CG2}.
We only describe the $\PO$-algebra structure on $\Obcl_{\cT}(U)$, since it is important to understand this structure to execute the proof of Theorem \ref{thm: main}.
Suppose that $\varphi \in B_{(k)}(U)$ and $\psi \in B_{(k')}(U)$.
Then,
\[
(\varpi_U)_{\Obcl_\cT}(\varphi, \psi) = \sum_{r\geq 0} (\varpi_U)_{\Obcl_\cT}^{(r)}(\varphi, \psi),
\]
where $(\varpi_U)_{\Obcl_\cT}^{(r)}(\varphi, \psi)$ is an element of $B_{(k+k'+r-2)}(U)$ defined by the following equation, where we implicitly choose symmetric-group representatives of $\delta \varphi, \delta\psi$, and $\Pi^{(r)}$ in order to evaluate them on elements of $\sL(U)$ (verifying in the process that the resulting $\bbS_{k+k'+r-2}$-orbit does not depend on these choices):
\begin{align}
&(\varpi_U)_{\Obcl_\cT}^{(r)}(\varphi, \psi)(s_1,\ldots, s_{k+k'+r-2})\nonumber\\
= \frac{1}{2}&\langle \delta \varphi(s_1,\ldots, s_{k-1}), \Pi^{(r)}(s_k, \ldots, s_{k+r-1}, \delta\psi(s_{k+r},\ldots, s_{k+k'+r-2}))\rangle_\del\nonumber\\
\label{eq: bdyPObracket}
+\frac{(-1)^{|\varphi||\psi|}}{2}&\langle \delta \psi(s_1,\ldots, s_{k'-1}), \Pi^{(r)}(s_{k'}, \ldots, s_{k'+r-1}, \delta\varphi(s_{k'+r},\ldots, s_{k+k'+r-2}))\rangle_\del
\end{align}
where $\ip_\del$ denotes the natural pairing between $\sL(U)$ and $\sL^!_c(U)$.
One may verify, using integration by parts and the fact that $\Pi^{(r)}$ is a polydifferential operator, that $\{\varphi,\psi\}^{(r)}$ has smooth first derivative again.
\end{proof}

\begin{remark}
Figure \ref{fig: PB} provides a graphical depiction of a typical term appearing in the Poisson bracket of two functionals $\varphi$ and $\psi$.
It is a depiction of the second term on the right-hand side of Equation \eqref{eq: bdyPObracket}.
The diagram is to be read from top to bottom, with the edges on top of a vertex to be understood as inputs, and the ones at the bottom of a vertex to be understood as outputs.
The arrows on the edges indicate field type.
An input arrow pointing in to its vertex is an element of $\sL(U)$; an input arrow pointing out of its vertex is an element of $\sL^!(U)$.
We take the opposite convention for output arrows.
The ``elbow'' connecting the $D\varphi$ and $\Pi^{(r)}$ vertices is to be interpreted as the natural pairing between $\sL^!_c$ and $\sL$.
As drawn, it seems that there is an assymmetry between the way the $\varphi$ and $\psi$ vertices appear in the Poisson bracket.
This is mostly just a relic of the way we chose to write down the bracket above.
There is an implicit ``modular'' nature to the $\varphi$ and $\psi$ vertices which allows us to exchange input and output arrows.
Using this modularity, we could redraw the diagram from Figure \ref{fig: PB} in a way to suggest that $\varpi(\varphi, \psi)$ is a double composite.

Alternatively, one could describe the Poisson bracket of functionals in terms of the Schouten-Nijenhuis bracket $[\cdot,\cdot]$, which Butson and Yoo define on the complex of polyvector fields with smooth first derivative. In that language, one may define
\[
(\varpi_U)_{\Obcl_\cT}(\varphi,\psi) := \pm[[\Pi,\psi],\varphi]\pm [[\Pi,\varphi],\psi];
\]
this bracket satisfies the Jacobi identity as a consequence of the equation $[\Pi,\Pi]=0$.
\end{remark}

\begin{figure}[h]
\centering
\begin{tikzpicture}[scale=1.5]
         \node at (-.25,0)(k){$\delta\varphi$};
         \draw[ultra thick] (0, 0) arc (0:30:.25cm) --+(0,0) node(a){};
         \draw[ultra thick] (a) arc (30:60:.25cm) --+(0,0) node(b){};
        \draw[ultra thick] (b) arc (60:120:.25cm) --+(0,0) node(c){};
        \draw[ultra thick] (c) arc (120:150:.25cm) --+(0,0) node(d){};
        \draw[ultra thick] (d) arc (150:180:.25cm) --+(0,0) node(e){};
        \draw[ultra thick] (e) arc (180:270: .25cm) --+(0,0) node(f){};
        \draw[ultra thick] (f) arc (270:360: .25cm) --+(0,0) node (g){};
        \draw[fermionbar] (a.center) -- +(30:.5) node (h){};
        \draw[fermionbar] (d.center) -- +(150:.5) node (i){};
        \draw[fermionbar] (b.center) -- +(60:.5) node (h){};
        \draw[fermionbar] (c.center) -- +(120:.5) node (i){};
        \path (k.center) -- +(0,.5) node{$\cdots$};
        
        \path (k.center) --+(0:1.5cm) node (k1){$\Pi^{(r)}$};
        \path (k1.center) --+(.25,0) node (z1){};
        \draw[ultra thick] (z1) arc (0:60:.25cm) --+(0,0) node (a1){};
        \draw[ultra thick] (a1) arc (60:80:.25cm) --+(0,0) node(b1){};
        \draw[ultra thick] (b1) arc (80:120:.25cm) --+(0,0) node(c1){};
        \draw[ultra thick] (c1) arc (120:150:.25cm) --+(0,0) node(d1){};
        \draw[ultra thick] (d1) arc (150:180:.25cm) --+(0,0) node(e1){};
        \draw[ultra thick] (e1) arc (180:270: .25cm) --+(0,0) node(f1){};
        \draw[ultra thick] (f1) arc (270:360: .25cm) --+(0,0) node (g1){};
        \draw[fermion] (a1.center) -- +(60:.5) node (h1){};
        \draw[fermionbar] (d1.center) -- +(150:.5) node (i1){};
        \draw[fermionbar] (b1.center) -- +(80:.5) node (h1){};
        \draw[fermionbar] (c1.center) -- +(120:.5) node (i1){};
        \path (k1.center) -- +(-0.10,.60) node{$\cdots$};
        \draw[fermionbar] (f.center) to[out = 270, in =270] (f1.center) node (j){};
        
        \path (k1.center) --+(60:1cm) node (k2){$\delta\psi$};
        \path (k2.center) --+(.25,0) node (z2){};
        \draw[ultra thick] (z2) arc (0:30:.25cm) --+(0,0) node (a2){};
        \draw[ultra thick] (a2) arc (30:60:.25cm) --+(0,0) node(b2){};
        \draw[ultra thick] (b2) arc (60:120:.25cm) --+(0,0) node(c2){};
        \draw[ultra thick] (c2) arc (120:150:.25cm) --+(0,0) node(d2){};
        \draw[ultra thick] (d2) arc (150:180:.25cm) --+(0,0) node(e2){};
        \draw[ultra thick] (e2) arc (180:240: .25cm) --+(0,0) node(f2){};
        \draw[ultra thick] (f2) arc (240:360: .25cm) --+(0,0) node (g2){};
        \draw[fermionbar] (a2.center) -- +(30:.5) node (h2){};
        \draw[fermionbar] (d2.center) -- +(150:.5) node (i2){};
        \draw[fermionbar] (b2.center) -- +(60:.5) node (h2){};
        \draw[fermionbar] (c2.center) -- +(120:.5) node (i2){};
        \path (k2.center) -- +(0,.5) node{$\cdots$};
         \end{tikzpicture}
         \caption{A graphical depiction of a typical term appearing in the Poisson bracket of two elements $\varphi,\psi\in \Obcl_{\cT}$.}
        \label{fig: PB}
\end{figure}
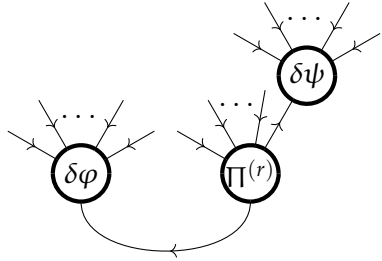

Let us now construct the $\PO$-prefactorization algebra $\Obcl_{\cZ(\cT),\cT}$ on $N\times \RRge$ which forms the ``bulk'' part of the bulk-boundary correspondence.

For the rest of this section (but not in the next section), given $U\subseteq N\times \RRge$ open, set
\[
\TA(U) = \cinfty(U; L\boxtimes \Lambda^\sharp T^* \RRge)\oplus \cinfty_D(U; L^!\boxtimes \Lambda^\sharp T^* \RRge).
\]
Here, the subscript $D$ in the second summand denotes that we consider only the $L^!$-valued forms which vanish when pulled back to the boundary $U\cap N\times \{0\}$.
The space $\TA(U)[-1]$ is an $L_\infty$-algebra with brackets determined as follows.
There are brackets $\ell_{k, B\to A}$ obtained from the original brackets on $\sL$ via tensoring with the commutative dg algebra $\Omega^\bullet(\RRge)$.
For $k=1$, this includes also the de Rham differential on $\Omega^\bullet(\RRge)$.
Such brackets also include the coadjoint $L_\infty$-brackets for the action of $\sL[-1]$ on $\sL^!$.
There are also brackets $\ell_{r,\Pi}$ which are induced from the Poisson structure.
We refer the reader to Definition 3.18 of \autocite{butsonyoo} for more details.

Note that there is a natural cohomological degree --1 pairing between $\TA(U)$ and $\TA_c(U)$; this pairing is invariant with respect to all brackets on $\TA(U)[-1]$, so that in particular there is an injective map of complexes $\TA_c(U)[1]\to \Hom_\CVS(\TA(U),\RR)$.
Moreover, this pairing induces a pairing of cohomological degree +1 on $\TA_c(U)[1]$.

The reader familiar with the BV formalism/bulk-boundary systems might notice that we are describing the structure of a BV theory on $\TA$.
More precisely, we are describing the space of fields of the universal bulk-boundary system for $\cT$.
The ``boundary'' part of the term ``bulk-boundary system'' appears here as the Dirichlet condition on $\sL^!$-valued forms.

Before we give the definition of the $\PO$-factorization algebra $\Obcl_{\cZ(\cT),\cT}$, let us apply these constructions to Example \ref{ex: cswzw}. Later in this section, we will revisit Example \ref{ex: scalarft}.

\begin{example}
Let us apply these constructions to the Poisson BV theory $\cT$ of Example \ref{ex: cswzw}. In this case, setting for simplicity of notation $U=U'\times I$, where $U'\subseteq \Sigma$ and $I\subseteq \RRge$ are open, we find
\[
\TA(U) = \Omega^{(0,\sharp)}_\Sigma(U')\hotimes_\beta \Omega^\sharp_D(I)\otimes \fg[1] \oplus \Omega^{(1,\sharp)}_\Sigma(U') \hotimes_\beta \Omega^\sharp(I)\otimes \fg^*.
\]
Let us suppose further that $\kappa$ is non-degenerate.
In this case, we may use $\kappa$ to identify $\fg^*$ with $\fg$. 
Then, the underlying graded space of $\TA(U)$ is identified with the space of $\fg$-valued forms in $U$ whose pullback to the boundary lies in the $(1,\sharp)$ forms on $\Sigma$.
We recognize this space as the space of fields of Chern-Simons theory with Wess-Zumino-Witten boundary condition (cf. Section 4.3 of \autocite{butsonyoo}).
Moreover, tracing through the constructions, the differential and 2-bracket on $\TA(U)[-1]$ induced from the Poisson structure of Example \ref{ex: cswzw} are precisely the de Rham differential and the usual bracket of $\fg$-valued forms. Finally, it is straightforward to verify that the pairing on~$\TA_c(U)[1]$ (when this space is understood as the space of $\fg$-valued forms on $U$ with a boundary condition) is precisely the pairing defining Chern-Simons theory as in \autocite{AKSZ} or \autocite{costrenormBV}.

On the other hand, consider the factorization algebra $\cF^{2\pi i \kappa}$ appearing in Section 5.5.1 of \autocite{CG1}.
This is a factorization algebra on $\Sigma$/
Its underlying graded object is (a mollified version) of the underlying graded object of $\Obcl_\cT$.
Moreover, $\cF^{2\pi i \kappa}(U)$ has a graded commutative product and a differential, and the failure of the differential to be a derivation for this product is precisely the $\PO$-bracket on~$\Obcl_\cT$.
Hence, we should think of the factorization algebra $\cF^{2\pi i \kappa}$ as a quantization of the $\PO$-prefactorization algebra~$\Obcl_\cT$.
Moreover, this quantization reproduces the Kac-Moody \emph{vertex} algebra at level $\kappa$ (Theorem 5.5.1 of \autocite{CG1}).

So, as applied to Example \ref{ex: cswzw}, Theorem \ref{thm: main} implies a correspondence between 1) classical Chern-Simons theory with Wess-Zumino-Witten boundary condition on $\Sigma\times \RRge$ and 2) the classical limit of the Kac-Moody vertex algebra on $\Sigma$.
It is an interesting question to determine whether this correspondence survives quantization, since in this case there is a candidate for the quantum observables of $\cT$.
\end{example}

Returning now to the general case, we must introduce---just as for $\Obcl_\cT$---functionals on $\TA$ with smooth first derivative.

\begin{notation}
The convenient vector space $A_{(k)}(U)$ is the subspace of 
\[
\Hom_{\CVS}(\TA(U)^{\hotimes k},\RR)_{\mathbb{S}_k}
\]
consisting of those functionals with \emph{smooth first derivative}, i.e. they consist of those functionals such that all of their representatives, when viewed as maps
\[
\Hom_\CVS(\TA(U)^{\hotimes_\beta (k-1)},\Hom_{\CVS}(\TA(U),\RR))
\]
factor through the inclusion $\TA_c(U)[1]\to \Hom_\CVS(\TA(U),\RR)$.
As above, if $\varphi$ has smooth first derivative, we may form the well-defined element 
\[
\delta\varphi \in \Hom_\CVS(\TA(U)^{\hotimes_\beta(k-1)},\TA_c(U)[1])_{\bbS_{k-1}},
\]
i.e. $\delta\varphi$ can be interpreted as the de Rham differential applied to $\varphi$.
\end{notation}

\begin{definition}[cf. Theorem 3.6 of \autocite{classicalarxiv}]
Let $\cT$ be a Poisson BV theory on $N$. The factorization algebra of \textbf{classical observables of the universal bulk-boundary system for $\cT$} is the $\PO$-factorization algebra defined by 
\[
\Obcl_{\cZ(\cT),\cT}(U) = \prod_{k\geq 0} A_{(k)}(U).
\]
The differential on $\Obcl_{\cZ(\cT),\cT}(U)$ is the Chevalley-Eilenberg differential for $\TA(U)[-1]$, and the commutative product is induced from the product in the symmetric algebra.
The structure maps are defined as for $\Obcl_{\cT}$, via extension-by-zero maps and multiplication in the commutative algebra $\Obcl_{\cZ(\cT),\cT}$.
The Poisson bracket on $\Obcl_{\cZ(\cT),\cT}(U)$ is defined as follows; given $\varphi\in A_{(k)}$ and~$\psi\in A_{(k')}$,
\[
(\varpi_U)_{\Obcl_{\cZ(\cT),\cT}}(\varphi, \psi) (s_1,\ldots, s_{k+k'-2}):=\langle \delta\varphi(s_1,\ldots, s_{k-1}), \delta\psi(s_{k},\ldots, s_{k+k'-2})\rangle,
\]
where $\langle \cdot, \cdot\rangle$ is the natural cohomological degree +1 pairing on $\TA_c(U)[1]$.
\end{definition}

\begin{example}
Consider the Poisson BV theory $\cT$ of Example \ref{ex: scalarft}. Again, set $U= U'\times I$ with $U'\subseteq N$ and $I\subseteq \RRge$ open.
Then,
\[
\xymatrix{
&&\cinfty_N(U')\hotimes_\beta \Omega^\sharp(I){\ar@(ur,ul)[]+U+<4pt,0pt>;{[]+U+<-4pt,0pt>}_{d_{dR}}}\ar[r]^(.45){\Delta_g}& \cinfty_N(U')\hotimes_\beta\Omega^\sharp(I){\ar@(ur,ul)[]+U+<4pt,0pt>;{[]+U+<-4pt,0pt>}_{d_{dR}}}[-1]\\
\TA(U)=&\Denssheaf_N(U')\hotimes_\beta \Omega^\sharp_D(I)[1]{\ar@(dr,dl)[]+D+<4pt,0pt>;{[]+D+<-4pt,0pt>}^{d_{dR}}}\ar[r]^(.53){\Delta_g}\ar@{>->}[ur]&\Denssheaf_N(U')\hotimes_\beta \Omega^\sharp_D(I){\ar@(dr,dl)[]+D+<4pt,0pt>;{[]+D+<-4pt,0pt>}^{d_{dR}}}\ar@{>->}[ur]&
};
\]
There are no higher brackets on $\TA(U)$.
If $0\notin I$, then the Dirichlet condition on the bottom row is vacuous, and the diagonal arrows are isomorphisms. It follows that $\TA(U)$ is acyclic, since it is a mapping cone of a cochain isomorphism. Hence, $\Obcl_{\cZ(\cT),\cT}(U)\simeq \RR$ as cochain complexes. In particular, the inclusion of the constant functionals into $\Obcl_{\cZ(\cT),\cT}(U)$ exhibits this equivalence; it is easy to check that this inclusion preserves the relevant $\PO$-algebra structures and the factorization products.
Hence, for Example \ref{ex: scalarft}, $\left.\Obcl_{\cZ(\cT),\cT}\right|_{N\times \RRgt}\simeq \RR$ as $\PO$-factorization algebras.

The argument we used in this example applies more generally to any Poisson BV theory $\cT$ where $\Pi^{(0)}$ is a zeroth-order differential operator induced from a cohomological degree +1 bundle isomorphism $L^!\to L$.
In particular, it applies to symplectic BV theories.
In other words, we see that for any such Poisson BV theory, the classical ``purely bulk'' observables are trivial.
The universal bulk theory $\cZ(\cT)$ therefore measures the extent to which the Poisson structure of $\cT$ fails to be symplectic.

On the other hand, Theorem \ref{thm: main} shows that the \emph{boundary} observables $\Obcl_{\cZ(\cT),\cT}$, where $\cT$ is a symplectic BV theory, reproduce the classical observables of $\Obcl_{\cT}$, as described in \autocite{CG2}.

\end{example}

\section{The Proof of the Main Result}
\label{sec: proof}
In this section, we prove Theorem \ref{thm: main}. We break the proof into parts for ease of digestion. 
First, we construct deformation retractions of $\rho_*\Obcl_{\cZ(\cT),\cT}(U)$ onto $\Obcl_\cT(U)$ for every open subset $U\subseteq N$. This is a two-step process.
Second, we transfer the $(\PO\text{-}\Disj_N)$-algebra structure on $\rho_*\Obcl_{\cZ(\cT),\cT}$ to a $\hoPODisj$-algebra structure on $\Obcl_\cT$, and verify that this $\hoPODisj$-algebra structure happens to be a $(\PO\text{-}\Disj_N)$-algebra structure.
Finally, we identify the transferred $(\PO\text{-}\Disj_N)$-algebra structure on $\Obcl_\cT$ with the one described in Definition \ref{deflem: BPOstruct}.
Before explaining these steps in detail, however, let us establish some notation.
We let $A= \rho_*\Obcl_{\cZ(\cT),\cT}$ and $B = \Obcl_{\cT}$, and we set
\[
\TA : = \sL\hotimes_\beta \Omega^\sharp(\RRge) \oplus \sL^!\hotimes_\beta \Omega^\sharp_D(\RRge)
\]
and 
\[
\TB : = \sL.
\]
These are both presheaves of $L_\infty$-algebras on $N$.
Note also that these definitions constitute a bit of a departure of notation from the previous section. 
The object we call $\TA$ here is technically the pushforward via $\rho$ of the sheaf we called $\TA$ in the previous section.

In $\TA$, we call elements of the first summand ``base-type fields'' and elements of the second summand ``fiber-type fields.''
The terminology arises from the fact that $\TA$ arises from a ``twisted cotangent bundle'' construction. 

We introduce (bi)-gradings on $A(U)$ and $B(U)$ which will be useful accounting tools.
The grading on $B$ is simply the overall polynomial degree: we will let $B_{(k)}$ denote the subspace of $B$ consisting of elements of polynomial degree $k$.
This notation coincides with the notation introduced in Definition \ref{def: smoothB}.
On $A$, we can refine the polynomial-degree grading by considering separately the degree of dependence on base-type and fiber-type fields.
Hence, we will let $A_{(k,\ell)}$ denote the space of functionals which accept $k$ base-type field inputs and $\ell$ fiber-type field inputs.

The differential on $B$ is a sum of terms $d_{(r)}$ ($r\geq 0$) which send~$B_{(k)}\to B_{(k+r)}$. The term $d_{(r)}$ is induced from the $(r+1)$-ary bracket on $\sL[-1]$.
We let $B^\tau$ denote the truncated complex whose underlying graded vector space is the same as that of $B$, but whose differential is just $d_{(0)}$.
This corresponds to considering the underlying abelian $L_\infty$-algebra of $\sL[-1]$.

A similar accounting happens for $A$: the differential consists of two terms $d^1$ and $d^2$, where $d^1$ is induced from the brackets $\ell_{k,B\to A}$ and $d^2$ from the brackets $\ell_{k,\Pi}$. 
The differential $d^2$ increases $\ell$-degree by exactly one, and reduces $k$-degree by at most 1. 
The differential $d^1$ preserves $\ell$-degree and is non-negative with respect to the $k$-degree.
We write $d^1_{(r)}$ for the term in $d^1$ which maps $A_{(k,\ell)}$ to $A_{(k+r,\ell)}$.
We write $d^2_{(r)}$ for the term in $d^2$ which maps $A_{(k,\ell)}$ to $A_{(k-1+r,\ell+1)}$.
(We have chosen the conventions so that $r$ must be non-negative in either case.)
As for $B$, we may consider the truncation $A^\tau$, where we consider only the term $d^1_{(0)}$ appearing in the differential. 
This corresponds to considering the underlying abelian $L_\infty$-algebra on $\sL[-1]$ with zero Poisson bivector.

Having introduced the notation, we can outline the proof of Theorem \ref{thm: main}.
First, in Lemma \ref{lem: firstdr}, we will construct a strong deformation retract $(\iota^\tau(U),\pi^\tau(U),\eta^\tau(U))$ of $A^\tau(U)$ onto $B^\tau(U)$ for each $U\subseteq N$.
Then, in Lemma \ref{lem: DR2}, we will ``turn on'' the remaining terms of the differential on $A$, and use the homological perturbation lemma to perturb this retract to a retract $(\iota(U),\pi(U),\eta(U))$ of $A(U)$ onto $B(U)$.
Next, in Lemma \ref{lem: transferedstructures}, we will verify that the homotopy transfer of the $(\PO\text{-}\Disj_N)$-algebra structure on $A$ along the retract of Lemma \ref{lem: DR2} induces a strict $(\PO\text{-}\Disj_N)$-algebra structure on $B$ (as opposed to a $\hoPODisj$-algebra structure).
Finally, in Lemma \ref{lem: finallem}, we verify that this $(\PO\text{-}\Disj_N)$-algebra structure on $B$ coincides with the one defined in Section \ref{subsec: PODisj}.

\begin{lemma}
\label{lem: firstdr}
For each open set $U\subseteq N$, there is a strong deformation retract 
\[
\xymatrix{
B^\tau(U) \ar@<1ex>[r]^-{\iota^\tau(U)} & A^\tau(U)\ar@<1ex>[l]^-{\pi^\tau(U)}{\ar@(ur,dr)[]+R+<0pt,4pt>;{[]+R+<0pt,-4pt>}^{\eta^\tau(U)}}
}.
\]
Moreover, these data satisfy
\begin{enumerate}
    \item The image of $\iota^\tau(U)$ is contained in~$\prod_{k\geq 0}A_{(k,0)}(U)$, and is a map of commutative algebras.
    \item The map $\pi^\tau(U)$ is zero on~$A_{(k,\ell)}(U)$ if~$\ell>0$, and is a map of commutative algebras.
    \item The map $\eta^\tau(U)$ preserves each~$A_{(k,\ell)}(U)$.
\end{enumerate}
\end{lemma}

\begin{proof}
Let us construct a strong deformation retract
\begin{equation}
\label{eq: tangentDR}
\xymatrix{
(\TB(U),\ell_{1,\sL}) \ar@<1ex>[r]^-{\di(U)}& (\TA(U),\ell_{1,B\to A})\ar@<1ex>[l]^-{\dpi(U)}{\ar@(ur,dr)[]+R+<0pt,4pt>;{[]+R+<0pt,-4pt>}^{\deta(U)}}
}.
\end{equation}
Given such a deformation retract, we can extend it to a strong deformation retract
\begin{equation}
\label{eq: symdr}
\xymatrix{
(\bigoplus_k\Hom_{\CVS}(\TB(U)^{\hotimes_\beta k},\RR)_{\mathbb{S}_k},\ell_{1,\sL}) \ar@<1ex>[r]^-{\iota_{\Sym}}& (\bigoplus_k \Hom_{\CVS}(\TA(U)^{\hotimes_\beta k},\RR)_{\mathbb{S}_k},\ell_{1,B\to A})\ar@<1ex>[l]^-{\pi_{\Sym}}{\ar@(ur,dr)[]+R+<0pt,4pt>;{[]+R+<0pt,-4pt>}^{\eta_{\Sym}}}
}
\end{equation}
using a standard trick (see, e.g. \autocite{othesis}).
Using this trick, $\iota_{\Sym}$ and $\pi_{\Sym}$ are algebra maps, from which the latter halves of statements 1) and 2) in the lemma follow.
This does not construct the desired deformation retract, since the algebras $A^\tau(U)$ and $B^\tau(U)$ are obtained respectively as subspaces of the complexes on the right-hand and left-hand sides of the above retraction.
So, we will have to check that the data $\iota_{\Sym},\pi_{\Sym}, \eta_{\Sym}$ descend to $A^\tau$ and $B^\tau$.

To this end, we remind the reader that $A^\tau$ and $B^\tau$ complexes of functionals with smooth first derivative.
The maps $\iota_{\Sym}, \pi_{\Sym}, $ and $\eta_{\Sym}$ are constructed by means of tensoring the transposes of the maps $\di(U), \dpi(U),$ and $\deta(U)$ with each other and with the identity.
Hence, if we show that these transpose maps preserve the subspaces of smooth linear functionals which appear in the definitions of $B^\tau$ and $A^\tau$, then the deformation retraction data $\iota_{\Sym}, \pi_{\Sym},\eta_{\Sym}$ will descend to corresponding data for $A^\tau$ and $B^\tau$.

So, to recap: we need to construct the retract in Equation \eqref{eq: tangentDR}, and show that its transpose preserves the subspace of smooth linear functionals.
To this end, choose a compactly-supported one-form $\alpha$ on $\RR_{\geq 0}$ such that $\int_{\RR_{\geq 0}}\alpha =1$.
Define
\begin{align*}
    \dpi(U) : \TA(U) &\to \TB(U)\\
    \dpi(U)(s_1,s_2)&= \int_{\RR_{\geq 0}} \alpha \wedge s_1;
\end{align*}
here, $s_1$ is a base-type field and $s_2$ is a fiber-type field.
To obtain a similar such quasi-isomorphism, one could have set $\dpi(U)(s_1,s_2)=\mathrm{ev}_0^*s_1$, where $\mathrm{ev}_0$ is the inclusion of the boundary point $\{0\}\subset\RRge$; however, this choice will not descend to the complex of functionals with smooth first derivative.
Note that $\dpi$ is zero on fiber-type fields, and requires $s_1$ to be a zero-form along $\RRge$.
Since $\iota_{\Sym}$ is constructed as the pullback via $\dpi$, the first of the assertions in the previous sentence implies statement (1) of the lemma.
Moreover, if one considers $s'\in \sL^!_c(U)$ as a linear functional on $\TB(U)$, then the transpose of $\dpi$ sends $s'$ to $s'\otimes \alpha$, which is a smooth linear functional on $\TA(U)$.
(Here, the check that $s'\otimes \alpha$ satisfies the boundary condition to lie in $\sL^!_c(U)\hotimes \Omega^\sharp_{D}(\RRge)$ is vacuous, since $\alpha$ is a one-form so pulls back to zero at the boundary of $\RRge$.)

Similarly, define
\begin{align*}
     \di(U) : \TB(U) &\to \TA(U)\\
   \di(U)(s)&= \mathrm{pr}_1^*s,
\end{align*}
where $\mathrm{pr}_1: U\times \RRge \to U$ is the projection onto the first factor. 
In other words, we extend $s$ to be a constant form in the normal direction.
The map $\pi_{\Sym}$ in Equation \eqref{eq: symdr} is induced from $\di(U)$ via transpose and symmetrization; since $\di$ has image in the base-type fields, $\pi_{\Sym}$ is zero on functionals which take in at least one fiber-type field.
This gives assertion (2) of the Lemma.
Given $s_1\in \sL_c\otimes \Omega^\sharp$ and $s_2\in \sL^!_c\otimes \Omega^\sharp_D$, considered as elements of the linear dual to $\TA(U)$, the transpose of $\di(U)$ sends $(s_1,s_2)$ to $\int_{\RRge} s_2$, which is also a smooth linear functional.
One verifies directly that $\dpi(U)\di(U)=\id$, and that these two maps respect the differentials.

Finally, we define $\deta(U)$, and show that it 1) preserves smooth linear functionals, 2) is a homotopy between $\di(U)\circ\dpi(U)$ and the identity, and 3) satisfies condition (3) of the lemma.
Consider the cohomological degree --1 endomorphism $I$ of $\Omega^\sharp(\RRge)$ which takes a one-form $\beta$ to the unique anti-derivative of $\beta$ which vanishes at the boundary $t=0$.
The map $I$ clearly preserves the space of forms which vanish when pulled back to the boundary.
Then, for $s_1$ a base-type field and $s_2$ a fiber-type field, we define
\[
\deta(U)(s_1,s_2) = ((1-\di\dpi)Is_1, Is_2).
\]
Here, we abuse notation slightly, and let $I$ refer to the tensor product of $I$ with the identity on $\sL(U)$ or $\sL^!(U)$, as appropriate.
Note that $[d_{dR},I] = \id-\mathrm{pr}_1^* \mathrm{ev}_0^*$,
where $\mathrm{ev}_0$ is the inclusion of the boundary point $\{\ast\}\hookrightarrow \RRge$.
From this it follows that $\deta(U)$ is a homotopy witnessing the exactness of $\di\dpi-\id$.
Moreover, it is straightforward to check using integration-by-parts on $\RRge$ that the transpose of $\deta$ preserves the smooth linear functionals.
Explicitly, if $s_1\in \sL_c(U)\hotimes_\beta \Omega^1_c(\RRge)$ and $s_2 \in \sL^!_c(U)\hotimes_\beta \Omega^1_{D,c}(\RRge)$, then
\begin{equation}
\label{eq: etaonmollified}
\left(\deta(U)\right)^T(s_1,s_2) = \left(I(s_1)-\int_{\RRge} s_1 ,-\left(\int_{\RRge} s_2 \right) \otimes I(\alpha) +I(s_2)\right).
\end{equation}

Note that $\deta(U)$ separately preserves the spaces of base-type and fiber-type fields.
And, the induced homotopy $\eta_{\Sym}$ is constructed by tensoring the identity with the transposes of $\deta$ and $\di\dpi$, all of which preserve both gradings on $A$. 
Hence, the induced deformation retract in Equation \eqref{eq: symdr} satisfies statement (3) of the Lemma.
Moreover, it is easy to verify directly that the deformation retraction $\di,\dpi,\deta$ is strong, and the ``strong-ness'' of a deformation retraction persists through the operations that produce $\iota^\tau, \pi^\tau, \eta^\tau$ from these.

This completes all verifications necessary for the Lemma.
\end{proof}

We may write the differential on $A(U)$ as $d^{1}_{(0)}+\delta_A$.
Since we have a strong deformation retraction as in Lemma \ref{lem: firstdr} which involves only $d^{1}_{(0)}$, we may ask whether we can ``perturb'' $\iota^\tau,\pi^\tau, \eta^\tau$, and $d_{(0)}$ into new data which give a strong deformation retraction of $A(U)$ onto $(B^\tau(U),d_{(0)}+\delta_B)$, for some $\delta_B$.
The homological perturbation lemma \autocite{crainic} gives a condition under which this is the case, and gives explicit formulas for a differential $d_{(0)}+\delta_B$ and a new deformation retraction $(\iota,\pi,\eta)$ of $A(U)$ onto $B(U)$ (with this new differential).
In the following lemma, we verify that this condition is satisfied, and we explore some elementary properties of the perturbed data.

\begin{lemma}
\label{lem: DR2}
For each open subset $U\subseteq N$, the deformation retraction $\iota^\tau(U),\pi^\tau(U),\eta^\tau(U)$ satisfies the hypotheses of the homological perturbation lemma \autocite{crainic}.
Application of the homological perturbation lemma induces a strong deformation retract
\[
\xymatrix{
B(U) \ar@<1ex>[r]^-{\iota(U)} & A(U)\ar@<1ex>[l]^-{\pi(U)}{\ar@(ur,dr)[]+R+<0pt,4pt>;{[]+R+<0pt,-4pt>}^{\eta(U)}}
};
\]
the maps $\pi(U)$ reamin unperturbed (i.e. $\pi(U)=\pi^\tau(U)$) and are therefore maps of commutative algebras.
\end{lemma}
\begin{proof}
We would like to apply the homological perturbation lemma to the deformation retraction constructed in Lemma \ref{lem: firstdr} and the perturbation $\delta_A:=\sum_{r\geq 1}d^1_{(r)}+\sum_{r\geq 0}d^2_{(r)}$.
To do this, we need to ensure that the perturbation is ``small'' in the sense that the infinite sum
\[
(1-\delta_A \eta^\tau)^{-1}= \sum_{p=0}^\infty (\delta_A\eta^\tau)^p
\]
converges.
To this end, the bigrading on $A^\tau$ and the properties (1)-(3) of Lemma \ref{lem: firstdr} are useful. We will show that, for a fixed $(k,\ell)$, only finitely many of the powers $(\delta_A\eta_{(0)})^p$ can give non-zero elements in $A_{(k,\ell)}$.
Explicitly, if $p> k+2\ell$, then the image of $(\delta_A\eta_{(0)})^p$ cannot have a nontrivial projection onto $A_{(k,\ell)}$. To see this, note that we can expand $\delta_A$ as a sum of two terms: one which strictly increases the sum of the two gradings (namely, $\sum_{r\geq 1}(d^1_{(r)}+d^2_{(r)})$), and one which preserves the sum, but increases the second grading by 1 (namely, $d^2_{(0)}$).
When we expand out $(\delta_A \eta^\tau)^p$ using this decomposition of $\delta_A$, the result will be a sum of terms with $s$ factors involving the first term in $\delta_A$ and $p-s$ factors involving the second term. 
The number $s$ can range from 0 to $p$.
Note that either $s> k+\ell$ or $p-s> \ell$.
Moreover, $\eta^\tau$ preserves $(k,\ell)$-bidegree, so if $s>k+\ell$, then the $s$ factors of $\left( \sum_{r\geq 1}d^1_{(r)}+d^2_{(r)}\right) \eta^\tau$ tell us that we need to ``come from'' $A_{(k',\ell')}$ with $k'+\ell'<0$ in order to ``hit'' $A_{(k,\ell)}$.
Similarly, if $p-s>\ell$, then the $p-s$ factors of $d^2_{(0)}\eta^\tau_{(0)}$ tell us that we need to ``come from'' $A_{(k+\ell', \ell-\ell')}$ for $\ell'>\ell$ in order to ``hit'' $A_{(k,\ell)}$. 
In either case, this is impossible.
This completes the proof that the homological perturbation lemma applies. Hence, we get a deformation retraction 
\begin{equation}
\label{eq: defretract2}
\xymatrix{
(B^\sharp(U),d_{(0)}+\delta_B) \ar@<1ex>[r]^-{\iota(U)} & A(U)\ar@<1ex>[l]^-{\pi(U)}{\ar@(ur,dr)[]+R+<0pt,4pt>;{[]+R+<0pt,-4pt>}^{\eta(U)}}
},
\end{equation}
where $\delta_B$ is the perturbation of $d_{(0)}$ induced from $\delta_A$.

We now need to check that $\delta_B$ is precisely the part of the Chevalley-Eilenberg differential which increases $\Sym$ degree, i.e. that $(B^\sharp , d_{(0)}+\delta_B) = B$.
Then, we will have indeed constructed a deformation retract of $A(U)$ onto $B(U)$.
The perturbation $\delta_B$ is given by the formula
\[
\delta_B = \pi^\tau\sum_{p=0}^\infty (\delta_A\eta^\tau)^p \delta_A \iota^\tau.
\]
Here again, keeping track of bigradings is useful.
The perturbation $\delta_A$ consists of $d^2$ and $\sum_{r\geq 1} d^1_{(r)}$.
The differential $d^2$ strictly increases the second, $\ell$-degree, while $d^{1}_{(r)}$ preserves the $\ell$-degree.
Since $\pi^\tau$ is only non-zero on factors $A_{(k,\ell)}$ where $\ell=0$ and $\eta^\tau$ preserves the bidigree, it follows that the formula for $\delta_B$ reduces to one in which $\delta_A$ is replaced by the term $\sum_{r\geq 1} d^1_{(r)}$.
Moreover, it is relatively straightforward to check that 
\begin{equation}
\label{eq: commdiffs}
\pi^\tau \left(\sum_{r\geq 1} d^1_{(r)}\right) = d_{CE}\pi^\tau,
\end{equation}
where $d_{CE}$ is the term in the Chevalley-Eilenberg differential on $B$ arising from all brackets of arity at least 2.
This follows from the fact that $\di: \TB[-1]\to \TA[-1]$, which induces $\pi^\tau$ via transpose and symmetrization, manifestly preserves the brackets.
Using Equation \eqref{eq: commdiffs}, the side condition $\pi^\tau\eta^\tau=0$, and the retraction equation $\pi^\tau\iota^\tau=\id$, we conclude that $\delta_B$ is precisely $d_{CE}$.

It remains to check that the $\pi^\tau(U)$ are unperturbed.
To this end, note that $\pi(U) = \pi^\tau \sum_{p\geq 1} (\delta_A \eta^\tau)^p$; note that $\pi^\tau \delta_A = \pi^\tau \left( \sum_{r\geq 1}d^1_{(r)}\right)$, and Equation \eqref{eq: commdiffs} together with the side condition $\pi^\tau \eta^\tau=0$ gives that $\pi(U)= \pi^\tau(U)$.
This is unsuprising, since as we have noted, $\pi^\tau(U)$ is induced from $\di(U)$, and $\di(U)$ is already a map of $L_\infty$-algebras (using all of the brackets on $\TA$ and $\TB$).
It follows by Lemma \ref{lem: firstdr}, Statement 2, that $\pi(U)$ is a map of dg commutative algebras for each $U$.
\end{proof}

\begin{lemma}
\label{lem: transferedstructures}
The $\hoPODisj$-algebra structure obtained on $B$ from the $(\PO\text{-}\Disj_N)$-algebra structure on $A$ and the strong deformation retracts of Lemma \ref{lem: DR2} is a $(\PO\text{-}\Disj_N)$-algebra structure.
\end{lemma}

\begin{proof}

As discussed in Section \ref{sec: PODisj}, a $\hoPODisj$-algebra structure on $B$ is described by a number of generating operations, one operation $\mu_T$ for each tree $T$ whose vertices are labelled by operations from $\PO\text{-}\Disj_N$.
Because we are transferring a strict $(\PO\text{-}\Disj_N)$-algebra structure from $A$ to $B$, the description of the operations $\mu_T$ obtained via homotopy transfer is relatively simple: we place $\iota$ on the leaves of $T$, the appropriate $\PO\text{-}\Disj_N$ operations from $A$ on the vertices of $T$, $\eta$ on the internal edges of $T$, and $\pi$ on the root.
Reading the operations from the leaves to the root, we obtain a multilinear operation on $B$.

To prove the lemma, it is therefore necessary and sufficient to show that $\mu_T=0$ if $T$ has more than one vertex.
To this end, we use two gradings on $A$ and $B$.
The first grading is by $\ell$-degree. 
The second is a new grading, which we call the $D$-degree and which we describe in a moment.
First, however, we remark that for both of these gradings, the entirety of $B$ is homogeneous of degree 0.
To prove the lemma, we will show that $\mu_T$ has non-zero degree for both of these gradings unless $T$ has only one vertex.

Now, to describe the $D$-degree.
Recall that $\TA$ has a tensor factor of $\Omega^\sharp(\RRge)$ in its definition.
For $D\geq 0$, define $A_{(k,\ell,D)}$ to be the subspace of $A_{(k,\ell)}$ which requires among its inputs $k+\ell-D$ fields which are zero-forms along $\RRge$.
In other words, elements of $A_{(k,\ell,0)}$ accept only zero-forms along $\RRge$ as inputs.
We define a similar grading on $B$ by setting $D= 0$ on $B$.
Evidently, 
\[
A_{(k,\ell)} = \bigoplus_{D=0}^{k+\ell} A_{(k,\ell,D)}.
\]
Note also that $\pi=\pi^\tau$ is zero on $A_{(k,\ell, D)}$ for $D>0$ (and also for $\ell>0$, we have seen in Lemma \ref{lem: firstdr}).
Table \ref{table: Ddegrees} shows how various other operations behave with respect to $D$-degree.
\begin{table}[h]
    \centering
    \begin{tabular}{|c|c|}
    \hline
         Operator & $D$-degree  \\
         \hline
       $\eta^\tau$  & +1\\
       $\iota^\tau$ & 0 \\
       $\mu_A$ & 0\\
       $\varpi_A$ & -1\\
       $d^1_{(r)}$ & 0\\
       $d^2_{(r)}$ & 0\\
       $m_{U_1,\ldots, U_k}^V$ & 0\\
       \hline
    \end{tabular}
    \caption{Behavior of Operators on $A$ with respect to $D$-degree.}
    \label{table: Ddegrees}
\end{table}
We will also need the explicit formulas for $\eta(U)$, $\iota(U)$, and $\pi(U)$ obtained from the homological perturbation lemma:
\begin{align*}
    \iota(U) = \left(\sum_{p=0}^\infty (\eta^\tau(U) \delta_A)^p\right)\iota^\tau(U), \quad \pi(U) = \pi^\tau(U), \quad \eta(U) = \eta^\tau(U) \left( \sum_{p=0}^\infty (\delta_A\eta^\tau(U))^p\right),
\end{align*}
where the simplification for $\pi$ is obtained in Lemma \ref{lem: DR2}.
As noted above, the transferred operations on $B$ are represented by trees where each vertex is labelled by a $\PO\text{-}\Disj_N$ operation in $A$, the leaves are labelled by $\iota$, and the root is labelled by $\pi=\pi^\tau$.
We may further decompose the operations represented by such trees as follows. 
We use the explicit formulas for the perturbed data, and we label each leaf by $\iota^\tau$ and each edge by $\eta^\tau$. Each leaf and edge also acquires a non-negative integer label, which labels the power of $(\delta_A \eta^\tau)$ appearing at the leaf or edge.
The transferred operation $(\mu_T)_B$ is a sum over all labellings of $T$ by these integers.
Let $T$ be a tree labeled thus, and let $r$ be the sum of the integer labels on $T$.
Note also that the $D$-degree of a vertex $v$ of $T$ must be non-positive since it is a composite of $\mu_A$, $\varpi_A$, and $m_{U_1,\ldots, U_k}^V$ (cf. Definition \ref{def: PODisj}); let $-D_v$ be this $D$-degree.
Now that $T$ is labeled thus, we can compute its $D$-degree.
Each edge of $T$ raises $D$-degree by 1 plus the integer label at the edge. 
The leaves raise the $D$-degree by the corresponding integer labels. 
Each vertex $v$ changes the $D$-degree by $-D_v$.
Combining these observations, we find that $\mu_T$ is non-zero only when
\begin{equation}
\label{eq: Ddegreematching}
|E(T)|+r -\sum_{v\in V(T)}D_v=0.
\end{equation}
Let us now consider $\ell$-degree, using the same notation $r$ and $D_v$ as before. Just as for $D$-degree, the trees need to have $\ell$-degree 0 to be non-zero operations on $B$. 
Recall that $\eta^\tau$, $\pi^\tau$, $\iota^\tau$, and $\mu_A$ preserve $\ell$-degree; $\varpi_A$ lowers $\ell$-degree by 1, and $\delta_A$ can increase $\ell$-degree by at most 1.
From this it follows that each eadge and leaf raises $\ell$-degree at most by the corresponding integer lable, and each vertex lowers $\ell$-degree by $D_v$.
Let $|T|_\ell$ denote the $\ell$-degree of $T$; we find, for $\mu_T$ to be non-zero, we need
\[
0=|T|_\ell \leq r-\sum_{v\in V(T)}D_v.
\]
Plugging Equation \eqref{eq: Ddegreematching} into this inequality, we obtain
\[
0\leq -|E(T)|,
\]
from which we obtain that $T$ has no edges, so consists of exactly one vertex.
\end{proof}

By Lemma \ref{lem: transferedstructures}, we find that $B$ acquires a $(\PO\text{-}\Disj_N)$-algebra structure.
It remains to check that this structure coincides with the structure described in Definition \ref{deflem: BPOstruct}.

\begin{lemma}
\label{lem: finallem}
The $(\PO\text{-}\Disj_N)$-algebra structure induced on $B$ by homotopy transfer from $A$ coincides with that of Definition \ref{deflem: BPOstruct}.
\end{lemma}

\begin{proof}
The operad $\PO\text{-}\Disj_N$ is generated by the operations $\mu_U$, $\varpi_U$, $m_{U}^V$, and $m_{U_1,U_2}^{U_1\sqcup U_2}$ (Lemma \ref{lem: PODisjgens}).
It suffices to verify that the operations on $B$ corresponding to these generators match.

We start with $m_U^V$. 
Recall the notation $D_v$ as in the proof of the preceding lemma.
In this case, $D_v=0$, so by Equation \eqref{eq: Ddegreematching}, $r=0$.
Hence, homotopy transfer from $A$ induces the following map on $B$:
\[
\pi^\tau(V) \circ (m_U^V)_A\circ \iota^\tau(U).
\]
The maps $(m_U^V)_A$ are pullbacks along the restriction maps of $\TA$ (as a sheaf on $N$).
Similarly for $B$.
The maps $\pi^\tau$ and $\iota^\tau$ are induced from the maps $\di$ and $\dpi$, which manifestly respect these restriction maps. It follows that 
\[
\pi^\tau(V) \circ (m_U^V)_A\circ \iota^\tau(U)= (m_U^V)_B,
\]
as desired.

Next, we consider the operation $\mu_U$.
Again, $r=0$, so that the transferred operation on $B$ is
\[
\pi^\tau(U) \circ (\mu_U)_A\circ(\iota^\tau(U)\otimes \iota^\tau(U)).
\]
Because $\pi^\tau$ is an algebra map, and because $\pi^\tau\iota^\tau = \id_B$, it follows that this operation is $\mu_B$.

Third, let us consider the operations $m_{U_1,U_2}^{U_1\sqcup U_2}$.
Again, $r=0$; moreover, since it happens that
\[
(m_{U_1,U_2}^{U_1\sqcup U_2})_A = (\mu_{U_1\sqcup U_2})_A \circ (m_{U_1}^{U_1\sqcup U_2})_A\otimes (m_{U_2}^{U_1\sqcup U_2})_A,
\]
this third case reduces to the previous two.

Finally, we consider the operations $\varpi_U$.
In this case, we have $D_v=1$, so $r=1$.
The transferred operations on $B$ are 
\[
\pi(U) (\varpi_U)_A(\iota(U)\otimes \iota(U));
\]
we have already remarked that $\pi(U)=\pi^\tau(U)$, and moreover when we expand $\iota(U)$ in powers of $(\delta_A\eta^\tau)$, only one such power is allowed between the two factors of $\iota(U)$ appearing in the formula.
Moreover, only the term $d^2$ in the differential on $A$ may appear in order to compensate for the fact that $\varpi_U$ lowers $\ell$-degree by 1.
So, it remains to check that 
\begin{equation}
\label{eq: bracket}
\pi^\tau\varpi_A(\eta^\tau d^2\iota^\tau  \otimes \iota^\tau)+ \pi^\tau\varpi_A(\iota^\tau\otimes \eta^\tau d^2\iota^\tau)
\end{equation}
coincides precisely with the $\PO$-bracket $\varpi_B$ (we suppress all dependence on $U$ from the notation for ease in reading).
This part of the proof very much resembles the proof of the quantum theorem in Section 4 of \autocite{GRW}.
The main difference is that the accounting is more complicated for a number of reasons; in both cases, however, the proof boils down to integrals of bump functions on $\RRge$.

Let us start by considering linear observables $\varphi,\psi \in \sL^!_c(U)$.
Then, for $s_1,\ldots, s_r\in \sL(U)$, $s'\in \sL^!(U)$, $\beta_1,\ldots, \beta_r\in \Omega^\sharp(\RRge)$, and $\beta\in \Omega^\sharp_D(\RRge)$, recall that we have
\[
\varpi_A^{(r)} (\eta^\tau d^2_{(r)}\iota^\tau\varphi, \iota^\tau \psi) = \langle \delta \eta^\tau d^2_{(r)}\iota^\tau\varphi, \delta \iota^\tau \psi\rangle;
\]
we also have
\[
\left(d^2_{(r)}\iota^\tau \varphi\right)(s_1\otimes \beta_1,\ldots, s_r\otimes \beta_r, s'\otimes \beta) =\pm\langle \varphi, \Pi^{(r)}(s_1,\ldots, s_r,s)\rangle_\del \int_{\RRge} \alpha\wedge \beta_1\wedge \cdots \wedge \beta_r\wedge \beta;
\]
Next, $\eta^\tau d^2_{(r)}\iota^\tau \varphi$ is a sum of terms. In the sum, we allow ourselves to apply $\deta$ to exactly one input of $d^2_{(r)}\iota^\tau \varphi$, and on the other inputs, we apply either $\di\dpi$ or $\id$ and sum over all possibilities with appropriate combinatorial weights.
For our computation, we will see that we only need to keep some of these summands, so we avoid spelling this sum out in full detail at this stage.
Since $\Pi^{(r)}$ is a differential operator, for each $1\leq i\leq r$, there is a differential operator 
\begin{equation}
\label{eq: formaladj}
\Pi^{(r),!_i}: \sL(U)^{\times(i-1)} \times \sL^!(U)\times \sL(U)^{\times(r-i)} \times \sL^!(U)\to \sL^!(U),
\end{equation}
obtained as (up to a sign) the formal adjoint to $\Pi^{(r)}$ in the $i$-th slot, satisfying
\begin{equation}
\label{eq: lastformaladj}
\langle \varphi, \Pi^{(r)}(s_1,\ldots, s_r,s)\rangle_\del = \langle \Pi^{(r),!_i}(s_1,\ldots, s_{i-1},\varphi, s_{i+1}, \ldots, s_{r},s),s_i\rangle_\del.
\end{equation}
Similarly, there is an operator $\Pi^{(r),!_{r+1}}$ for the last slot. 
In other words, 
\begin{align*}
\delta_i \left( d^{2}_{(r)} \iota^\tau \varphi\right) &(s_1\otimes \beta_1,\ldots, s_{r-1}\otimes \beta_{r-1},s\otimes \beta)\\
&= \pm \Pi^{(r),!_i}(s_1,\ldots, s_{i-1},\varphi, s_{i},\ldots, s_{r-1},s) \otimes \alpha\wedge \beta_1\wedge \cdots \wedge \beta_{r-1}\wedge \beta
\end{align*}
for $1\leq i\leq r$, and 
\begin{align*}
\delta_{r+1}\left( d^{2}_{(r)} \iota^\tau \varphi\right) &(s_1\otimes \beta_1,\ldots, s_{r}\otimes \beta_r)\\
&= \pm\Pi^{(r),!_{r+1}}(s_1,\ldots, s_{r},\varphi) \alpha\wedge \beta_1\wedge \cdots \wedge \beta_{r}.
\end{align*}
Equations \eqref{eq: etaonmollified}, \eqref{eq: formaladj}, and \eqref{eq: lastformaladj} can be combined into a formula for $\delta_i \left( \eta^\tau d^2_{(r)}\iota^\tau \varphi\right)$.
This will be a rather lengthy formula, and since many of the terms in it drop out when we perform the desired calculation, we will not write it out explicitly.
Instead, we will make some observations that will lead to an explicit formula for \[\pi^\tau\varpi_A(\eta^\tau d^2_{(r)} \iota^\tau \varphi \otimes \iota^\tau \psi).\]
Note that since $\Pi^{(r),!_i}$ for $1\leq i\leq r$ has output in $\sL^!$, and $\eta^\tau$ will not change this, only the term $\delta_{r+1}\eta^\tau d^2_{(r)}\iota^\tau \varphi$ will be able to pair non-trivially with $\delta\iota^\tau \psi$.
Finally, when we apply $\pi^\tau$, we will set all $\beta_i$'s to 1, so that $\di\dpi=\id$ when applied to any input of this form.
When we combine all of these observations, we find:
\begin{align*}
\pi^\tau\varpi_A&(\eta^\tau d^2_{(r)} \iota^\tau \varphi \otimes \iota^\tau \psi)(s_1,\ldots, s_r)= \\
&=\pm \langle \Pi^{(r),!_{r+1}}(s_1,\ldots, s_r, \varphi),\psi\rangle_\del \int_{\RRge} \left(1-I(\alpha)\right)\alpha\\
&= \pm\frac{1}{2} \langle \varphi , \Pi^{(r)}(s_1,\ldots, s_r,\psi)\rangle_\del.
\end{align*}
This is one of the two terms in Equation \eqref{eq: bdyPObracket}.
The other term is obtained by considering the second term in Equation \eqref{eq: bracket}.
It can be verified that the overall signs appearing here match those in the definition of $\varpi_B$, Equation \eqref{eq: bdyPObracket}.
Finally, note that the argument is structurally identical, though more cumbersome in terms of notation, for $\varphi\in B_{(k)}$ and $\psi\in B_{(k')}$ with $k,k'>1$.
\end{proof}

\addtocontents{toc}{\vspace{.5em}} 

\label{Bibliography}
\printbibliography
\end{document}

%% file: macros.tex
 \usepackage{mathpple}
\usepackage{mathrsfs}
\usepackage{amsmath, amscd, amsthm,amssymb, thmtools, amsfonts}
\usepackage[mathcal]{eucal}
\usepackage{relsize}
\usepackage{tikz-cd}
\usepackage[all]{xy}
\usepackage{graphicx}
\usepackage{subcaption}
\usepackage{epstopdf}
\usepgfmodule{shapes}
\usepackage{color}
\usepackage{ifthen}

\renewcommand{\epsilon}{\varepsilon}

\newcommand{\CC}{\mathbb C}

\newcommand{\PP}{\mbb P}

\newcommand{\ip}[1][\cdot,\cdot]{\left\langle #1 \right\rangle}

\newcommand{\DVS}{\mathrm{DVS}}
\newcommand{\CVS}{\mathrm{CVS}}

\newcommand{\PO}{\PP_0}

\newcommand{\POne}{\PP_1}
\newcommand{\hoPODisj}{(\widehat{\PO\text{-}\Disj_N})_\infty}

\DeclareMathOperator{\Sym}{Sym} \DeclareMathOperator{\Hom}{Hom}

\def\cC{\mathcal C}
\def\cF{\mathcal F}

\def\cT{\mathcal T}

\def\cZ{\mathcal Z}

\def\CC{\mathbb C}
\def\HH{\mathbb H}

\def\PP{\mathbb P}
\def\RR{\mathbb R}\def\bbS{\mathbb S}

\def\ZZ{\mathbb Z}

\def\sC{\mathscr C}
\def\sE{\mathscr E}
\def\sL{\mathscr L}
\def\sO{\mathscr O}\def\sP{\mathscr P}

\def\sV{\mathscr V}

\def\fJ(E){\mathfrak E}
\def\fJ{\mathfrak J}

\def\fX{\mathfrak X}
\def\fg{\mathfrak g}

\declaretheoremstyle[
spaceabove=7pt, spacebelow=7pt,
headfont=\normalfont\bfseries,
notefont=\mdseries, notebraces={(}{)},
bodyfont=\normalfont,
postheadspace=5pt,
headpunct = .
]{thm}

\declaretheoremstyle[
spaceabove=7pt, spacebelow=7pt,
headfont=\normalfont\bfseries,
notefont=\mdseries, notebraces={(}{)},
bodyfont=\normalfont,
postheadspace=10pt,
headpunct = .
]{def}

\declaretheoremstyle[
spaceabove=4pt, spacebelow=7pt,
headfont=\itshape,
postheadspace=5pt,
headpunct = :,
postheadspace = 3pt, qed = $\lozenge$
]{rem}

\declaretheorem[numbered = yes, parent = section, style = thm]{theorem}

\declaretheorem[numbered=yes, style = thm]{conjecture}

\declaretheorem[sibling = theorem, style = thm, name = Theorem/Definition]{thm-def}

\declaretheorem[sibling = theorem, style = thm]{lemma}
\declaretheorem[sibling = theorem, style = thm]{notation}

\declaretheorem[sibling = theorem, style = thm, name = Definition-Lemma]{deflem}
\numberwithin{equation}{section}

\declaretheorem[sibling = theorem, style = def]{definition}

\declaretheorem[sibling = theorem, style = rem]{remark}

\declaretheorem[sibling = theorem, style = rem]{example}

\newcommand{\cinfty}{C^{\infty}}
\newcommand\Obcl[1][~]{\ifthenelse{ \equal{#1}{~}} {
  \operatorname{Obs}^{cl}
}{
  \operatorname{Obs}^{cl}(#1)
}}
\newcommand\Obq[1][~]{\ifthenelse{ \equal{#1}{~}} {
  \operatorname{Obs}^{q}
}{
  \operatorname{Obs}^{q}(#1)
}}

\newcommand{\coshriek}{\scriptstyle \text{\rm !`}}

\DeclareMathOperator{\Disj}{Disj}

\DeclareMathOperator{\id}{id}

\newcommand{\sEb}{\sE_\partial}

\newcommand{\Dens}{\mathrm{Dens}}
\newcommand{\Denssheaf}{\mathscr{D}ens}

\def\del{\partial}

\def\RRge{\RR_{\geq 0}}
\def\RRgt{\RR_{> 0}}

\def\hotimes{{\widehat{\otimes}}}

\usepackage{epigraph}

\newcommand{\TA}{\mathrm{T}A}
\newcommand{\TB}{\mathrm{T}B}
\newcommand{\di}{\mathrm d \iota^\tau}
\newcommand{\dpi}{\mathrm d \pi^\tau}
\newcommand{\deta}{\mathrm d \eta^\tau}

\usepackage{varwidth}
\usetikzlibrary{decorations.pathmorphing}
\usetikzlibrary{decorations.markings}
\usetikzlibrary{snakes}
\pgfarrowsdeclare{chevron}{chevron}{...}
{
\pgfsetdash{}{0pt} 
\pgfsetlinewidth{.5pt}
\pgfpathmoveto{\pgfpoint{-1pt}{0pt}}
\pgfpathlineto{\pgfpoint{-2pt}{2pt}}
\pgfpathlineto{\pgfpoint{0pt}{2pt}}
\pgfpathlineto{\pgfpoint{1pt}{0pt}}
\pgfpathlineto{\pgfpoint{0pt}{-2pt}}
\pgfpathlineto{\pgfpoint{-2pt}{-2pt}}
\pgfpathclose
\pgfusepathqstroke
}
\pgfarrowsdeclare{newdiamond}{newdiamond}{...}
{
\pgfsetdash{}{0pt} 
\pgfsetlinewidth{.5pt}
\pgfpathmoveto{\pgfpoint{0pt}{0pt}}
\pgfpathlineto{\pgfpoint{-1pt}{2pt}}
\pgfpathlineto{\pgfpoint{1pt}{2pt}}
\pgfpathlineto{\pgfpoint{0pt}{0pt}}
\pgfpathlineto{\pgfpoint{1pt}{-2pt}}
\pgfpathlineto{\pgfpoint{-1pt}{-2pt}}
\pgfpathclose
\pgfusepathqstroke
}

\tikzset{
    vector/.style={decorate, decoration={snake}, draw},
	provector/.style={decorate, decoration={snake,amplitude=2.5pt}, draw},
	antivector/.style={decorate, decoration={snake,amplitude=-2.5pt}, draw},
    fermion/.style={draw=black, postaction={decorate},
        decoration={markings,mark=at position .55 with {\arrow{>}}}},
    Ahalfedge/.style={draw=black, postaction={decorate},
        decoration={markings,mark=at position .25 with {\arrow{>}}, mark = at position 1 with {\arrow[draw=black,xshift=1.5pt]{Circle[length=3pt]}}}},
    Bhalfedge/.style={draw=black, postaction={decorate},
        decoration={markings,mark=at position .80 with {\arrow[draw=black]{>}}, mark = at position 0 with {\arrow[draw=black,xshift=1.5pt]{Circle[length=3pt]}}}},
 ABedge/.style={draw=black, postaction={decorate},
        decoration={markings,mark=at position .50 with {\arrow[draw=black]{chevron}}}},
ABfulledge/.style={draw=black, postaction={decorate},
        decoration={markings,mark = at position .50 with {\arrow[draw=black]{chevron}}, mark = at position .3 with {\arrow[draw=black]{>}}, mark = at position .75 with {\arrow[draw=black]{>}}, mark = at position .0 with {\arrow[draw=black,xshift=1.5pt]{Circle[length=3pt]}}, mark = at position 1 with {\arrow[draw=black,xshift=1.5pt]{Circle[length=3pt]}}}},
BBfulledge/.style={draw=black, postaction={decorate},
        decoration={markings,mark = at position .50 with {\arrow[draw=black]{newdiamond}}, mark = at position .3 with {\arrow[draw=black]{>}}, mark = at position .75 with {\arrow[draw=black]{<}}, mark = at position .0 with {\arrow[draw=black,xshift=1.5pt]{Circle[length=3pt]}}, mark = at position 1 with {\arrow[draw=black,xshift=1.5pt]{Circle[length=3pt]}}}},
 BBedge/.style={draw=black, postaction={decorate},
        decoration={markings,mark=at position .50 with {\arrow[draw=black]{newdiamond}}}},
    fermionbar/.style={draw=black, postaction={decorate},
        decoration={markings,mark=at position .55 with {\arrow{<}}}},
    fermionnoarrow/.style={draw=black},
    gluon/.style={decorate, draw=black,
        decoration={coil,amplitude=4pt, segment length=5pt}},
    scalar/.style={dashed,draw=black, postaction={decorate},
        decoration={markings,mark=at position .55 with {\arrow[draw=black]{>}}}},
    scalarbar/.style={dashed,draw=black, postaction={decorate},
        decoration={markings,mark=at position .55 with {\arrow[draw=black]{<}}}},
    scalarnoarrow/.style={dashed,draw=black},
    electron/.style={draw=black, postaction={decorate}/,
        decoration={markings,mark=at position .55 with {\arrow[draw=black]{>}}}},
	bigvector/.style={decorate, decoration={snake,amplitude=4pt}, draw},
	    ray/.style={draw=black, postaction={decorate},
        decoration={markings,mark=at position 1 with {\arrow[draw=black]{>}}, mark = at position 0 with {}}}
}

%% file: abstract.tex
\begin{abstract}
In this article, we use the language of $\mathbb{P}_0$-factorization algebras to articulate a classical bulk-boundary correspondence between 1) the observables of a Poisson Batalin-Vilkovisky (BV) theory on a manifold $N$ and 2) the observables of the associated universal bulk-boundary system on $N\times \RR_{\geq 0}$.
The archetypal such example is the Poisson BV theory on $\RR$ encoding the algebra of functions on a formal Poisson manifold, whose associated bulk-boundary system on the upper half-plane is the Poisson sigma model.
In this way, we obtain a generalization and justification of the basic insight that led Kontsevich to his deformation quantization of Poisson manifolds. The proof of these results relies significantly on the operadic homotopy theory of $\mathbb{P}_0$-algebras.
\end{abstract}